\documentclass[sn-mathphys]{sn-jnl}

\usepackage{amsmath,latexsym}
\usepackage{amsthm}
\usepackage{bm}
\usepackage{color}

\theoremstyle{thmstyleone}%
\newtheorem{theorem}{Theorem}
\newtheorem{lemma}{Lemma}%

\theoremstyle{thmstyletwo}%

\theoremstyle{thmstylethree}%

\begin{document}

\title[Physisorbate-Layer Problem]{The Physisorbate-Layer Problem
Arising in Kinetic Theory of Gas-Surface Interaction}


\author*[1]{\fnm{Kazuo} \sur{Aoki}}\email{kazuo.aoki.22v@st.kyoto-u.ac.jp}

\author[2]{\fnm{Vincent} \sur{Giovangigli}}\email{vincent.giovangigli@polytechnique.edu}

\author[3]{\fnm{Fran\c{c}ois} \sur{Golse}}\email{francois.golse@polytechnique.edu}

\author[4]{\fnm{Shingo} \sur{Kosuge}}\email{kosuge.shingo.6r@kyoto-u.ac.jp}

\affil[1]{\orgdiv{Department of Mathematics}, \orgname{National Cheng Kung University},
\orgaddress{
\city{Tainan} \postcode{70101},
\country{Taiwan}}}

\affil[2]{\orgdiv{CMAP, CNRS}, \orgname{\'{E}cole Polytechnique}, \orgaddress{
\city{91128 Palaiseau}, \postcode{Cedex},
\country{France}}}

\affil[3]{\orgdiv{CMLS}, \orgname{\'{E}cole Polytechnique}, \orgaddress{
\city{91128 Palaiseau}, \postcode{Cedex}, 
\country{France}}}

\affil[4]{\orgdiv{Institute for Liberal Arts and Sciences}, \orgname{Kyoto University}, \orgaddress{
\city{Kyoto} \postcode{606-8501}, 
\country{Japan}}}


\abstract{
A half-space problem of a linear kinetic equation for gas molecules
physisorbed close to a solid surface, 
relevant to a kinetic model of gas-surface interactions and derived by
Aoki {\it et al.} (K.~Aoki {\it et al.}, in: Phys. Rev. E 106:035306, 2022),
is considered.
The equation contains a confinement potential in the vicinity of the solid
surface and an interaction term between gas molecules and phonons.
It is proved that a unique solution exists when the incoming molecular
flux is specified at infinity. This validates the natural observation
that the half-space
problem serves as the boundary condition for the Boltzmann equation.
It is also proved that the sequence of 
approximate solutions used for the existence proof converges exponentially
fast. In addition, numerical results showing the details
of the solution to the half-space problem are presented.
}




\maketitle

\section{Introduction}\label{sec:intro}



The boundary condition for the Boltzmann equation results from
complex gas-surface interactions and specifies scattering kernels relating
the incident and reflected molecular fluxes at the surface.
The most conventional boundary condition is the Maxwell-type
condition, which is a linear combination of specular and diffuse reflection \cite{C88,S07}.
In addition to it, more general boundary conditions
have been proposed \cite{E67,K71,CL71,C88,L91,S13}.
However, most of these boundary conditions are of mathematical
or empirical nature and are not directly related to physical properties,
such as the characteristics of the gas and surface molecules and interaction
potentials. A more physical approach would be to use molecular dynamics
simulations to understand the relation between the incident and reflected molecular
fluxes \cite{WGK94,MM94,YTH06,WSQL21,CGLB23}. However, although useful for
assessing the existing boundary conditions,
this approach is in general not helpful in the construction of new models.

An alternative physical approach is the {\it kinetic} approach based on
kinetic equations describing the behavior of gas molecules interacting
with the surface molecules
\cite{BKPK86,BKP88,BBK94,BDKKS95,FG08,ACD11,BCM14,BCM16,AGH16,AG19a,AG19b,AG21,TAH21}.
The typical kinetic equations include
a potential generated by fixed crystal molecules and a collision term
with phonons describing the fluctuating part of the potential
of the crystal molecules. 

In a recent paper \cite{AGK22}, a kinetic model of gas-surface interactions,
which follows the line of
\cite{BKPK86,BKP88,BBK94,BDKKS95,FG08,ACD11,BCM14,BCM16,AGH16,AG19a,AG19b,AG21},
was proposed and was used to construct the boundary
condition for the Boltzmann equation. The model contains a confinement potential
in the vicinity of the solid surface, which produces a thin layer of physisorbed
molecules (physisorbate layer), as well as the interaction term between gas molecules
(the Boltzmann collision term) and that between gas molecules and phonons,
where the term ``gas molecules'' is also used for physisorbed gas molecules.
Under the assumptions that (i) the gas-phonon interaction is much more frequent
than the gas-gas interaction inside the physisorbate layer; (ii) the thickness of the
physisorbate layer is
much smaller than the mean free path of the gas molecules; and (iii) the gas-phonon
interaction is described by a simple collision model of relaxation type,
an asymptotic analysis was performed, and
a linear kinetic equation for the physisorbate layer was derived together with
its boundary condition at infinity. The resulting kinetic equation and the boundary
condition at infinity form a half-space problem, in which no boundary condition is
imposed on the solid surface because the confinement potential prevents
the gas molecules from reaching the surface.

Suppose that the half-space problem has a unique solution when the
velocity distribution function of the gas molecule toward the surface
is assigned at infinity. It is a natural assumption based on numerical
computation as well as physical considerations \cite{AGK22}. 
This means that the outgoing velocity distribution
of the gas molecules at infinity is determined by the
incident distribution towards the surface there.
Since the thickness of the physisorbate
layer is much smaller than the mean free path, the infinity in the scale of
the layer can be regarded as the surface of the solid wall in the scale of the mean
free path. Therefore, the half-space problem plays the role of the boundary
condition on the surface for the Boltzmann equation that is valid outside
the physisorbate layer.

In our previous paper \cite{AGK22}, the half-space problem for the physisorbate layer
mentioned above was also solved 
numerically, and the results provided the numerical evidence of the existence
and uniqueness of the solution. In addition, based on an iteration scheme
and its first iteration, an analytic model of the boundary condition for the Boltzmann
equation was constructed, and the numerical assessment of the model
showed its effectiveness.

The first aim of the present study is to rigorously prove the existence and uniqueness
of the solution to the half-space problem for the physisorbate layer
studied numerically and approximately in \cite{AGK22}. In addition,
the sequence of iterative approximate solutions, which is used for the
existence proof, is shown to converge exponentially fast with
respect to the number of iterations. The second aim is to
investigate the behavior of the solution to the
half-space problem numerically.
Since special interest was put on the Boltzmann
boundary condition in \cite{AGK22}, attention was focused not on the
behavior of the
solution itself but on the relation between the incoming and outgoing
molecular fluxes at infinity. In this paper, we put more attention
to the details of the solution in the physisorbate layer and show some
related numerical results.

The paper is organized as follows. The kinetic model for
gas-surface interactions proposed in \cite{AGK22} and the resulting
half-space problem for the physisorbate layer are summarized in
Sec.~\ref{sec:kinetic-layer}.
Section \ref{sec:math} is
devoted to rigorous proof of mathematical properties, such as the
existence and uniqueness of the solution, for the half-space problem.
In Sec.~\ref{sec:numerical}, some numerical results showing the behavior of
the solution to the half-space problem are presented.
Concluding remarks are given in Sec.~\ref{sec:remarks}.

\section{Kinetic equations and physisorbate layer}\label{sec:kinetic-layer}

In this section, we summarize the kinetic model of gas-surface interactions
proposed in \cite{AGK22} and the resulting half-space problem
for the physisorbate layer.

\subsection{Kinetic model}\label{subsec:kin-model}

We consider a single monatomic gas in a half space ($z>0$)
interacting with a plane crystal surface located at $z=0$,
where $\bm{x} = (x, y, z)$ indicates the space coordinates.
We assume that the gas molecules are subject to an interaction
potential $W$ generated by the {\it fixed} crystal molecules.
The interaction potential $W$ is assumed to depend only
on the normal coordinate $z$ for simplicity and is written
in the form
\begin{align}
W (z) = W_\text{s} (z/\delta)
= W_\text{s} (\zeta),
\end{align}
where $\delta$
is a characteristic range of the surface potential and
$\zeta = z/\delta$ denotes the rescaled normal coordinate,
which is dimensionless.
The rescaled potential  $W_\text{s}$ is such that 
\begin{equation}\label{potential}
\lim_{\zeta\to0} W_\text{s} (\zeta) = +\infty,
\qquad\quad
\lim_{\zeta\to+\infty} W_\text{s} (\zeta) = 0,
\end{equation}
and usually involves an attractive zone and a repulsing zone as
Lennard-Jones potentials integrated over all crystal molecules, as
illustrated in Fig.~\ref{fig1}.
To be more specific, we assume the following:
\begin{itemize}
\item[(i)] the potential $W_\text{s} (\zeta)$ is a smooth function of $\zeta$
and has a single minimum $W_\text{min} (< 0)$
at $\zeta=\zeta_\text{min} (>0)$, i.e. $W_\text{min} = W_\text{s} (\zeta_\text{min})$;
\item[(ii)] in the interval $(0, \zeta_\text{min})$, $W_\text{s} (\zeta)$ decreases from
$+\infty$ to $W_\text{min}$ monotonically, so that $(0, \zeta_\text{min})$
is the repulsive zone;
\item[(iii)] in the interval $(\zeta_\text{min}, \infty)$,
$W_\text{s} (\zeta)$ increases from $W_\text{min}$ to 0
monotonically, so that $(\zeta_\text{min}, \infty)$ is the attractive zone;
\item[(iv)] in the repulsive zone $(0, \zeta_\text{min})$, $W_\text{s} (\zeta)$ is convex
downward.
\end{itemize}
\noindent
The gas molecules that are trapped by the potential well
are called the physisorbed molecules and the set of such
molecules forms the physisorbate.

\begin{figure}[htb]
\begin{center}
\includegraphics[width= 0.57\textwidth]{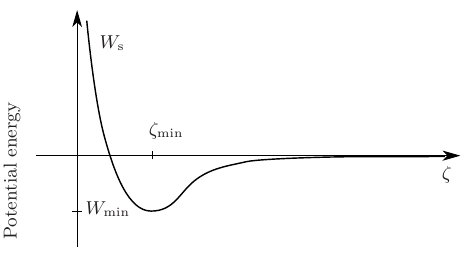}
\caption{\label{fig1}
Typical surface interaction potential $W_\text{s}$ as function of
$\zeta$.}
\end{center}
\end{figure}

The behavior of the gas is assumed to be governed by the
following kinetic equation of Boltzmann type
\cite{BKPK86,BKP88,BDKKS95,ACD11,AG19a,AG19b,AG21,AGH16,AGK22}:
\begin{equation}
\label{BE-dim}
\frac{\partial f}{\partial t} + \bm{c} \bm{\cdot} \frac{\partial f}{\partial \bm{x}} 
-
\frac{1}{m} \frac{\text{d} W}{\text{d} z} \frac{\partial f}{\partial c_z}
= J (f, f) + J_\text{ph} (f),
\end{equation}
where $m$ is the mass of a gas molecule,  $f(t, \bm{x}, \bm{c})$
is the velocity distribution function for
the gas molecules (including the physisorbed molecules), $t$ is the time
variable, $\bm{c}=(c_x, c_y, c_z)$ is the velocity of the gas molecules
with $c_x$, $c_y$, and $c_z$ being its $x$, $y$, and $z$ components, $J(f, f)$
is the Boltzmann collision operator describing the gas-gas collision,
and $J_\text{ph} (f)$ is the gas-phonon collision operator.

Since the explicit form of the Boltzmann collision operator $J(f, f)$
is not relevant to the present paper, it is omitted here. In \cite{AGK22},
under the assumption that the phonons are in equilibrium,
the following simple model of relaxation type is used for the gas-phonon
collision operator $J_\text{ph} (f)$:
\begin{equation}\label{Jph-dim}
J_\text{ph} (f) = \frac{1}{\tau_\text{ph}} \big( n M - f \big).
\end{equation}
Here, $\tau_\text{ph}$ is the relaxation time of gas-phonon interactions
and $n$ and $M$ are, respectively, the molecular number density
and the wall Maxwellian given by
\begin{subequations}\label{n-M}
\begin{align}
& n = \int_{\mathbf{R}^3} f \, \text{d} \bm{c},
\label{n-M-a} \\
& M = \left( \frac{m}{2 \pi k_\text{B} T_\text{w}} \right)^{3/2}
\exp \left( - \frac{m \vert \bm{c} \vert^2}{2 k_\text{B} T_\text{w}}
\right),
\label{n-M-b}
\end{align}
\end{subequations}
where $k_\text{B}$ is the Boltzmann
constant, $T_\text{w}$ is the temperature of the solid wall, and the
domain of integration in \eqref{n-M-a} is the whole space of $\bm{c}$.
The model \eqref{Jph-dim} was inspired by \cite{BKPK86,BKP88}
and was used in \cite{ACD11,AG19a,AG19b,AG21,AGH16,AGK22}.
We further assume that $\tau_\text{ph}$ has the same length scale of variation
as the potential $W$ and is a function of the scaled normal coordinate $\zeta$,
i.e.,
\begin{align}
\tau_\text{ph} (z) = \tau_\text{ph,s} (z/\delta)
= \tau_\text{ph,s} (\zeta).
\end{align}
Since there is no interaction between molecules and phonons 
far from the surface, we naturally assume that \cite{AGK22}
\begin{align}\label{tauph-limit}
\lim_{z\to\infty} \tau_\text{ph}(z)
= \lim_{\zeta\to\infty} \tau_\text{ph,s}(\zeta) = \infty.
\end{align}
It is also natural to assume that

\begin{itemize}
\item[(v)]
$\tau_{\text{ph,s}} (\zeta)$
is an increasing function of $\zeta$ in $[0, +\infty)$ with a finite
positive $\tau_{\text{ph,s}} (0)$;
\item[(vi)]
$1/\tau_\text{ph,s}(\zeta)$ is integrable over $[0, +\infty)$.
\end{itemize}

Since $J_\text{ph}$ as well as the potential $W$ vanishes
far from the surface, we may let $z \to \infty$ (or $\zeta \to \infty$) 
in \eqref{BE-dim} to obtain the kinetic equation in the gas phase
\begin{equation}
\label{BE-gas-dim}
\frac{\partial f_\text{g}}{\partial t} 
+ \bm{c} \bm{\cdot} \frac{\partial f_\text{g}}{\partial \bm{x}} 
= J (f_\text{g}, f_\text{g}),
\end{equation}
where $f_\text{g} (t, \bm{x}, \bm{c})$ denotes the velocity distribution
function of gas molecules. Equation \eqref{BE-gas-dim} is the standard
Boltzmann equation for a monatomic gas, and the distribution $f$
must converge to $f_\text{g}$ far from the surface.

\subsection{Normalization and parameter setting}\label{subsec:normalization}

In order to nondimensionalize the kinetic equation \eqref{BE-dim},
we introduce characteristic quantities that are marked with the 
${}^\star$ superscript. We denote by 
$n^\star$ the characteristic number density, 
$c^\star = (k_\text{B} T_\text{w}/m)^{1/2}$ the characteristic thermal speed,
$f^\star = n^\star/c^{\star 3}$ the characteristic molecular velocity distribution,
$\tau^\star_\text{fr}$ the characteristic mean free time,
$\lambda^\star = \tau^\star_\text{fr} c^\star$ the characteristic mean free path,
$W^\star = m c^{\star 2} = k_\text{B} T_\text{w}$ the characteristic potential, and
$\tau_\text{ph}^\star$ the characteristic time for gas-phonon interaction.
We recall that $\delta$ is the distance normal to the surface where the potential
$W$ is significant, so that $\tau_\text{la}^\star = \delta/c^\star$ indicates
the corresponding characteristic time of transit through the potential.

With these characteristic quantities, we introduce the
dimensionless quantities
$\hat{t}$, $\hat{\bm{x}}$, $\hat{\bm{c}}$, $\hat{n}$, $\hat{f}$,
$\hat{f}_\text{g}$, $\hat{M}$, $\hat{W}$, $\hat{\tau}_\text{ph}$, and $\hat{W}_\text{min}$,  
which correspond to 
$t$, $\bm{x}$, $\bm{c}$, $n$, $f$, $f_\text{g}$, $M$, $W$,
$\tau_\text{ph}$, and $W_\text{min}$,
respectively,
by the following relations:
\begin{align}\label{var-dimless}
\begin{aligned}
& \hat{t} = t/\tau_\text{fr}^\star, \qquad
\hat{\bm{x}} = \bm{x}/\lambda^\star, \qquad
\hat{\bm{c}} = \bm{c}/c^\star, \qquad \hat{n} = n/n^\star,
\\
& \hat{f} = f \, c^{\star 3}/n^\star, \qquad
\hat{f}_\text{g} = f_\text{g} \, c^{\star 3}/n^\star, \qquad
\hat{M} = M c^{\star 3},
\\
&
\hat{W} (\zeta) = W (z)/W^\star = W_\text{s} (\zeta)/k_\text{B} T_\text{w},
\quad
\hat{\tau}_\text{ph} (\zeta) = \tau_\text{ph} (z)/\tau_\text{ph}^\star
= \tau_\text{ph,s} (\zeta)/\tau_\text{ph}^\star,
\\
& \hat{W}_\text{min} = W_\text{min}/k_\text{B} T_\text{w}.
\end{aligned}
\end{align}
Correspondingly, the collision operators $J(f, f)$
and $J_\text{ph} (f)$ are nondimensionalized as
\begin{align}\label{J-dimless}
J(f, f) = \frac{n^\star}{\tau_\text{fr}^\star c^{\star 3}} \hat{J} (\hat{f}, \hat{f}), 
\end{align}
and
\begin{align}\label{Jph-dimless}
J_\text{ph} (f) = \frac{n^\star}{\tau_\text{ph}^\star c^{\star 3}} \hat{J}_\text{ph} (\hat{f}),
\qquad
\hat{J}_\text{ph} (\hat{f}) = \frac{1}{\hat{\tau}_\text{ph}} (\hat{n} \hat{M} - \hat{f}),
\end{align}
where
\begin{subequations}\label{n-M-dimless}
\begin{align}
& \hat{n} = \int_{\mathbf{R}^3} \hat{f} \text{d} \hat{\bm{c}},
\label{n-M-dimless-a} \\
& \hat{M} = (2\pi)^{-3/2} \exp \bigl( - \vert \hat{\bm{c}} \vert^2/2 \bigr),
\label{n-M-dimless-b}
\end{align}
\end{subequations}
and the domain of integration is the whole space of $\hat{\bm{c}}$.
The explicit form of $\hat{J} (\hat{f}, \hat{f})$, which is not relevant in
this paper, is omitted.

Substituting \eqref{var-dimless}--\eqref{n-M-dimless} into \eqref{BE-dim}, we
obtain the dimensionless version of \eqref{BE-dim}, which is characterized by
the following two dimensionless parameters: 
\begin{equation}\label{two-epsilons}
\epsilon_\text{ph} = \frac{\tau_\text{ph}^\star}{\tau^\star_\text{fr}},
\qquad
\epsilon = \frac{\delta}{\lambda^\star}
= \frac{\tau_\text{la}^\star}{\tau^\star_\text{fr}}.
\end{equation}
The kinetic scaling introduced in \cite{AGK22} reads
\begin{align}\label{epsilon-assump}
\epsilon_\text{ph} = \epsilon \ll 1.
\end{align}
This means that the effective range $\delta$ of the potential, which
is also the effective range of the gas-phonon interactions, is much
shorter than the characteristic mean free path $\lambda^\star$.
Therefore, the molecules trapped by the potential and interacting with the
phonons form a thin layer, which may be called the physisorbate
layer, in the scale of the mean free path. Equation \eqref{epsilon-assump}
also indicates that the characteristic time for gas-phonon interactions
$\tau_\text{ph}^{\star}$ is the same as the transit time across the layer
$\tau_\text{la}^{\star}$ and is much smaller than the
mean free time $\tau_\text{fr}^{\star}$.
The parameter setting \eqref{epsilon-assump} may be the simplest
{\it kinetic scaling} for the present model of the physisorbate layer.
This differs from the {\it fluid scaling} used in the derivation of
fluid-type boundary conditions \cite{AGH16,AG19a,AG19b,AG21}.    

In summary, we obtain the dimensionless version of \eqref{BE-dim}
in the following form:
\begin{align}
\label{BE-dimless}
\frac{\partial \hat{f}}{\partial \hat{t}}
+ 
\hat{\bm{c}}_\parallel \bm{\cdot} \frac{\partial \hat{f}}{\partial \hat{\bm{x}}_\parallel}
+
\hat{c}_z \frac{\partial \hat{f}}{\partial \hat{z}}
-
\frac{1}{\epsilon}
\frac{\text{d} \hat{W}}{\text{d} \zeta} \,
\frac{\partial \hat{f}}{\partial \hat{c}_z}
=
\frac{1}{\epsilon}
\hat{J}_\text{ph} (\hat{f})
+
\hat{J} (\hat{f}, \hat{f}),
\end{align}
where $\hat{\bm{c}}_\parallel = (\hat{c}_x, \hat{c}_y)$ and
$\hat{\bm{x}}_\parallel = (\hat{x}, \hat{y})$.
By taking the limit $\zeta \to \infty$ of \eqref{BE-dimless}, we recover
the dimensionless version of the Boltzmann equation \eqref{BE-gas-dim}
in the gas phase, i.e.
\begin{align}\label{BE-gas-dimless}
\frac{\partial \hat{f}_\text{g}}{\partial \hat{t}} 
+ \hat{\bm{c}}\, \bm{\cdot} \frac{\partial \hat{f}_\text{g}}{\partial \hat{\bm{x}}}
= \hat{J} (\hat{f}_\text{g}, \hat{f}_\text{g}).
\end{align}
%

\subsection{Physisorbate layer and boundary condition for Boltzmann equation}
\label{subsec:physisorbate}

In order to investigate the physisorbate layer, we assume
\begin{align}
\hat{f} = \hat{f} (\hat{t}, \hat{\bm{x}}_\parallel, \zeta, \hat{\bm{c}}_\parallel, \hat{c}_z),
\end{align}
because $\zeta = z/\delta = \hat{z}/\epsilon$ is the appropriate
normal coordinate for the layer. Then, \eqref{BE-dimless} is recast as
\begin{align}
\label{BE-layer}
\epsilon \left( \frac{\partial \hat{f}}{\partial \hat{t}}
+ 
\hat{\bm{c}}_\parallel \bm{\cdot} \frac{\partial \hat{f}}{\partial \hat{\bm{x}}_\parallel} \right)
+
\hat{c}_z \frac{\partial \hat{f}}{\partial \zeta}
-
\frac{\text{d} \hat{W}}{\text{d} \zeta} \,
\frac{\partial \hat{f}}{\partial \hat{c}_z}
=
\hat{J}_\text{ph} (\hat{f})
+
\epsilon \hat{J} (\hat{f}, \hat{f}).
\end{align}

This form suggests that $\hat{f}$ and thus $\hat{f}_\text{g}$ be expanded as
$\hat{f} = \hat{f}^{\langle 0 \rangle} + O(\epsilon)$ and 
$\hat{f}_\text{g} = \hat{f}_\text{g}^{\langle 0 \rangle} + O(\epsilon)$, respectively.
It is obvious that $\hat{f}_\text{g}^{\langle 0 \rangle}$ is also governed
by the Boltzmann equation \eqref{BE-gas-dimless}.
In the following, we consider only the zeroth order terms in $\epsilon$ and
identify $\hat{f}^{\langle 0 \rangle}$ and $\hat{f}_\text{g}^{\langle 0 \rangle}$
with $\hat{f}$ and $\hat{f}_\text{g}$, respectively (or equivalently, we omit the
superscript $\langle 0 \rangle$).
From \eqref{BE-layer}, \eqref{Jph-dimless}, and \eqref{n-M-dimless},
the equation for the zeroth order is obtained as
\begin{align}
\label{BE-physisorbate}
\hat{c}_z\, \frac{\partial \, \hat{f}}
{\partial \zeta}
-
\frac{\text{d} \hat{W} (\zeta) }{\text{d} \zeta} \, 
\frac{\partial \, \hat{f}}{\partial \hat{c}_z}
=
\frac{1}{\hat{\tau}_\text{ph} (\zeta)} (\hat{n} \hat{M} - \hat{f}),
\end{align}
where $\hat{n}$ and $\hat{M}$ are given by \eqref{n-M-dimless-a}
and \eqref{n-M-dimless-b}, respectively. Note that $\hat{n}$ here
is the zeroth-order number density in the physisorbate layer.
Equation \eqref{BE-physisorbate} is the kinetic equation
governing the physisorbate layer that will be investigated in the following.

Integrating both sides of \eqref{BE-physisorbate} with respect to $\hat{\bm{c}}$
over the whole space, we have
$(\partial/\partial \zeta)\int_{{\bf{R}}^3} \hat{c}_z \hat{f} \text{d}\hat{\bm{c}}=0$,
which leads to
\begin{align}\label{cons-c}
\int_{{\bf{R}}^3} \hat{c}_z \hat{f} \text{d}\hat{\bm{c}}=0,
\end{align}
because $\hat{f} \to 0$ as $\zeta \to 0$. This indicates the particle conservation.

As discussed in \cite{AGK22}, the connection condition between
the inner physisorbate layer and the outer gas domain at the zeroth order
is given by
\begin{align}\label{connection}
\hat{f} (\hat{t}, \hat{\bm{x}}_\parallel, \zeta \to \infty, \hat{\bm{c}}_\parallel, \hat{c}_z)
= \hat{f}_\text{g} (\hat{t}, \hat{\bm{x}}_\parallel, \hat{z}=0, \hat{\bm{c}}_\parallel, \hat{c}_z).
\end{align}
Note that $\hat{f}_\text{g}$ and thus \eqref{BE-gas-dimless} have been
extended to the crystal surface $\hat{z} = 0$.
Condition \eqref{connection} means
physically that the outer edge of the inner physisorbate layer
may be identified with the solid surface for the outer gas domain
(see \cite{AGK22} for a more quantitative argument).

As pointed out in \cite{AGK22}, \eqref{BE-physisorbate} is likely to have
a unique solution when $\hat{f}$ for the molecules toward the surface
($\hat{c}_z < 0$) is imposed at infinity, that is,
\begin{align}\label{cond-inf}
\hat{f} (\hat{t}, \hat{\bm{x}}_\parallel, \zeta, \hat{\bm{c}}_\parallel, \hat{c}_z)
\to \hat{f}_\infty (\hat{t}, \hat{\bm{x}}_\parallel, \hat{\bm{c}}_\parallel, \hat{c}_z),
\quad \text{as} \;\;\; \zeta \to \infty, \;\;\; \text{for} \;\;\; \hat{c}_z < 0,
\end{align}
where $\hat{f}_\infty$ is an arbitrary function of $\hat{t}$, $\hat{\bm{x}}_\parallel$,
$\hat{\bm{c}}_\parallel$, and $\hat{c}_z$ consistent with $\hat{f}_\text{g}$
at $\hat{z}=0$. This property, which has been confirmed
numerically in \cite{AGK22}, will be established mathematically in Sec.~\ref{sec:math}. 

Thus, the kinetic equation \eqref{BE-physisorbate} with
the boundary condition \eqref{cond-inf} determines the solution $\hat{f}$
and thus $\hat{f}$ for $\hat{c}_z > 0$ at infinity, and this constitutes
the physisorbate-layer problem. This means that the solution
defines the operator $\Lambda$ that maps $\hat{f}$ for $\hat{c}_z < 0$
to $\hat{f}$ for $\hat{c}_z > 0$ at infinity, i.e.,
\begin{align}\label{albedo-1}
\hat{f} (\hat{t}, \hat{\bm{x}}_\parallel, \zeta \to \infty, \hat{\bm{c}}_\parallel, \hat{c}_z > 0)
= \Lambda
\hat{f} (\hat{t}, \hat{\bm{x}}_\parallel, \zeta \to \infty, \hat{\bm{c}}_\parallel, \hat{c}_z < 0),
\end{align}
or equivalently, because of \eqref{connection},
\begin{align}\label{albedo-2}
\hat{f}_\text{g} (\hat{t}, \hat{\bm{x}}_\parallel,  \hat{z}=0, \hat{\bm{c}}_\parallel, \hat{c}_z > 0)
= \Lambda
\hat{f}_\text{g} (\hat{t}, \hat{\bm{x}}_\parallel,  \hat{z}=0, \hat{\bm{c}}_\parallel, \hat{c}_z < 0).
\end{align}
This relation indicates that the operator $\Lambda$ provides the boundary
condition for the Boltzmann equation on the surface $\hat{z}=0$.

\subsection{Half-space problem for the physisorbate layer}\label{subsec:half-space}

Equation \eqref{BE-physisorbate} and boundary condition \eqref{cond-inf}
form a boundary-value problem in the half space $\zeta> 0$. In this subsection,
the problem will be transformed into some different forms for later convenience.
Since the variables $\hat{t}$ and $\hat{\bm{x}}_\parallel$ are just the
parameters, we will omit them hereafter.

Here, we simplify some notations for convenience in the mathematical
arguments in Sec.~\ref{sec:math}. To be more specific, we omit the hat $\hat{}$ for the
dimensionless variables and the subscript ph of $\hat{\tau}_\text{ph}$,
that is,
\begin{align}\label{notation-ch}
(\hat{f},\, \hat{\bm{c}}_\parallel,\, \hat{c}_z,\, \hat{W},\,
\hat{\tau}_\text{ph},\, \hat{n},\, \hat{M},\, \hat{f}_\infty,\, \hat{W}_\text{min}) 
\Rightarrow
(f,\, \bm{c}_\parallel,\, c_z,\, W,\, \tau,\, n,\, M,\, f_\infty,\, W_\text{min}).
\end{align}
No confusion is expected with these changes. Then, the half-space
problem \eqref{BE-physisorbate} and \eqref{cond-inf} reads as follows:
\begin{subequations}\label{HS-1}
\begin{align}
& c_z\, \frac{\partial \, f}{\partial \zeta} -
\frac{\text{d} W (\zeta) }{\text{d} \zeta} \, 
\frac{\partial \, f}{\partial c_z}
= \frac{1}{\tau (\zeta)} (n M - f),
\label{HS-1-a} \\
& n = \int_{\mathbf{R}^3} f \text{d} \bm{c},
\label{HS-1-b} \\
& M = (2\pi)^{-3/2} \exp \left( - \vert \bm{c} \vert^2/2 \right),
\label{HS-1-c} \\
& f \to f_\infty (\bm{c}_\parallel, c_z), \;\;\; \text{for}\;\; c_z < 0,
\;\;\; \text{as} \;\; \zeta \to \infty.
\label{HS-1-d}
\end{align}
\end{subequations}

If we introduce the marginal
\begin{align}\label{marginal}
F(\zeta, c_z) = \int_{-\infty}^\infty \int_{-\infty}^\infty
f(\zeta, \bm{c}_\parallel, c_z) \text{d} c_x \text{d} c_y,
\end{align}
and integrate \eqref{HS-1} with respect to $c_x$ and $c_y$ each from
$-\infty$ to $\infty$, we obtain the following half-space problem
for $F(\zeta, c_z)$:
\begin{subequations}\label{HS-2}
\begin{align}
& c_z\, \frac{\partial \, F (\zeta, c_z)}{\partial \zeta} -
\frac{\text{d} W (\zeta) }{\text{d} \zeta} \, 
\frac{\partial \, F (\zeta, c_z)}{\partial c_z}
= \frac{1}{\tau (\zeta)} [n (\zeta) \mathcal{M} (c_z) - F (\zeta, c_z) ],
\label{HS-2-a} \\
& n (\zeta) = \int_{-\infty}^\infty F (\zeta, c_z) \text{d} c_z,
\label{HS-2-b} \\
& \mathcal{M} (c_z) = (2\pi)^{-1/2} \exp \left( - c_z^2/2 \right),
\label{HS-2-c} \\
& F (\zeta, c_z) \to F_\infty (c_z), \;\;\; \text{for}\;\; c_z < 0,
\;\;\; \text{as} \;\; \zeta \to \infty,
\label{HS-2-d}
\end{align}
\end{subequations}
where $F_\infty (c_z) = \int_{-\infty}^\infty \int_{-\infty}^\infty
f_\infty (\bm{c}_\parallel, c_z) \text{d}c_x \text{d}c_y$.

Now, let us put
\begin{align}\label{def-varepsilon}
\varepsilon = \frac{1}{2} c_z^2 + W (\zeta).
\end{align}
Then, for each $\varepsilon \in [W_\text{min}, \infty)$,
the range of $\zeta$ is as follows:
\begin{align}\label{range-zeta}
\begin{cases}
\ [\zeta_a (\varepsilon),\ \infty) \quad \text{for} \;\; \varepsilon \ge 0, \\
\ [\zeta_a (\varepsilon),\ \zeta_b (\varepsilon)] \quad
\text{for} \;\; W_\text{min}\le \varepsilon<0,
\end{cases}
\end{align}
where $\zeta_a (\varepsilon)$ is the solution of $\varepsilon=W (\zeta)$
for $\varepsilon \ge 0$, and $\zeta_a (\varepsilon)$ and $\zeta_b (\varepsilon)$
are the two solutions of the same equation satisfying
$\zeta_a (\varepsilon)
\le \zeta_\text{min} \le \zeta_b (\varepsilon)$
for $W_\text{min} \le \varepsilon<0$.
The locations of $\zeta_a (\varepsilon)$ and $\zeta_b (\varepsilon)$
are shown schematically in Fig.~\ref{fig2}.

\begin{figure}[htb]
\begin{center}
\includegraphics[width= 0.90\textwidth]{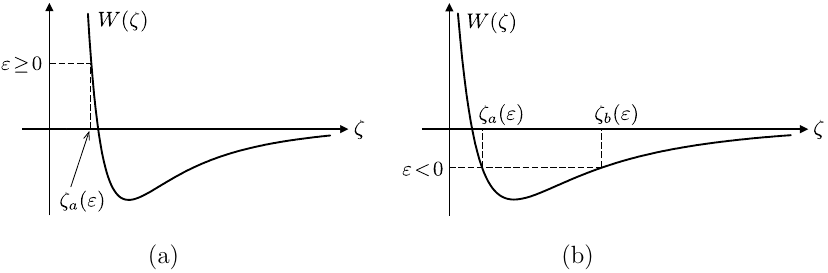}
\caption{\label{fig2}
The locations of $\zeta_a (\varepsilon)$ and $\zeta_b (\varepsilon)$.
(a) $\varepsilon \ge 0$, (b) $W_\text{min} \le \varepsilon<0$.
}
\end{center}
\end{figure}

Using \eqref{def-varepsilon}, we transform the independent variables from
$(\zeta, c_z)$ to $(\zeta, \varepsilon)$ and define
\begin{align}
F_\pm (\zeta, \varepsilon)
= F \left( \zeta, \pm \, \sqrt{2 [\varepsilon-W(\zeta)] } \right).
\end{align}
Note that $F_+$ corresponds to $c_z > 0$, and $F_-$ to $c_z < 0$.
Then, the problem \eqref{HS-2} is transformed to
\begin{subequations}\label{HS-3}
\begin{align} 
& \pm \sqrt{2 [\varepsilon-W(\zeta)]}\,
\frac{\partial\, F_\pm (\zeta, \varepsilon)}{\partial \zeta} 
= \frac{1}{\tau (\zeta)} \left[ 
n (\zeta) \mathcal{M} (\zeta, \varepsilon)  - F_\pm (\zeta, \varepsilon) \right],
\label{HS-3-a} \\ 
& n (\zeta) = \int_{W(\zeta)}^\infty
[ F_- (\zeta, \varepsilon) + F_+ (\zeta, \varepsilon) ]
\frac{\text{d}\varepsilon}{\sqrt{2 [\varepsilon-W(\zeta)]}},
\label{HS-3-b} \\
& \mathcal{M} (\zeta, \varepsilon) = (2\pi)^{-1/2} 
\exp \left( - \varepsilon + W (\zeta) \right), 
\label{HS-3-c} \\
& F_- (\zeta, \varepsilon) \to F_\infty (-\sqrt{2\varepsilon}), \;\;\; \text{for}\;\; \varepsilon > 0,
\;\;\; \text{as} \;\; \zeta \to \infty,
\label{HS-3-d}
\end{align}
\end{subequations}
where $\mathcal{M} (\zeta, \varepsilon)$ is the expression
of $\mathcal{M}(c_z)$ in \eqref{HS-2-c} in terms of $\zeta$ and $\varepsilon$.
In addition, we need to assume the continuity condition
\begin{subequations} \label{HS-3-continuity}
\begin{align}
& F_+ \left( \zeta_a (\varepsilon), \varepsilon \right)  
= F_- \left( \zeta_a (\varepsilon), \varepsilon \right),
\quad  \text{for}\ \ \varepsilon > W_\text{min},
\label{HS-3-continuity-a} \\
& F_- \left( \zeta_b (\varepsilon), \varepsilon \right)  
= F_+ \left( \zeta_b (\varepsilon), \varepsilon \right),
\quad  \text{for}\ \ W_\text{min} < \varepsilon < 0,
\label{HS-3-continuity-b}
\end{align}
\end{subequations}
at $\zeta=\zeta_a (\varepsilon)$ and $\zeta_b (\varepsilon)$ to complete the
half-space problem in the new variables $(\zeta, \varepsilon)$.
The conditions \eqref{HS-3-continuity-a} and \eqref{HS-3-continuity-b} are natural
because the molecules with energy $\varepsilon$ reaching the points
$\zeta = \zeta_a (\varepsilon)$  and $\zeta_b (\varepsilon)$
stop ($c_z=0$) there and then change the direction of motion.

In the problem \eqref{HS-2} or \eqref{HS-3}, the dependence on $c_x$ and
$c_y$ has been averaged out because of \eqref{marginal}. However, once
the solution $F (\zeta, c_z)$ or
$F_\pm (\zeta, \varepsilon)$ is obtained, the density $n (\zeta)$ is
also known. Therefore, \eqref{HS-1-a} reduces to a PDE, which can be solved
in response to the boundary condition \eqref{HS-1-d}. Therefore,
the problem \eqref{HS-1} and the problem \eqref{HS-2} or \eqref{HS-3}
are equivalent. Here, we note that
$[F_\pm (\zeta, \varepsilon),\, W (\zeta),\, n (\zeta),\,
\mathcal{M} (\zeta, \varepsilon),\, F_\infty (-\sqrt{2 \varepsilon}),\, \tau (\zeta)]$
in this subsection are equal to
$[\phi_\pm (\zeta, \varepsilon),\, \hat{\textsc{w}} (\zeta),\, \hat{\mathsf{n}} (\zeta),\,
\hat{\textsc{m}}_\text{m} (\zeta, \varepsilon),\, \phi_\infty (-\sqrt{2 \varepsilon}),\,
\hat{\tau}_\text{ph} (\zeta)]$ in Sec.~VI A in \cite{AGK22}.

\subsection{Iteration scheme}\label{subsec:iteration}

On the basis of \eqref{HS-3}, the following iteration scheme is defined:
\begin{subequations}\label{HS-3-it}
\begin{align} 
& \pm \sqrt{2 [\varepsilon-W(\zeta)]}\,
\frac{\partial\, F_\pm^k (\zeta, \varepsilon)}{\partial \zeta} 
= \frac{1}{\tau (\zeta)} \left[ 
n^{k-1} (\zeta) \mathcal{M} (\zeta, \varepsilon)  - F_\pm^k (\zeta, \varepsilon) \right],
\label{HS-3-it-a} \\ 
& n^k (\zeta) = \int_{W(\zeta)}^\infty
[ F_-^k (\zeta, \varepsilon) + F_+^k (\zeta, \varepsilon) ]
\frac{\text{d}\varepsilon}{\sqrt{2 [\varepsilon-W(\zeta)]}},
\label{HS-3-it-b} \\
& F_-^k (\zeta, \varepsilon) \to F_\infty (-\sqrt{2\varepsilon}), \;\;\; \text{for}\;\; \varepsilon > 0,
\;\;\; \text{as} \;\; \zeta \to \infty,
\label{HS-3-it-c} \\
& F_+^k \left( \zeta_a (\varepsilon), \varepsilon \right)  
= F_-^k \left( \zeta_a (\varepsilon), \varepsilon \right),
\quad  \text{for}\ \ \varepsilon > W_\text{min},
\label{HS-3-it-d} \\
& F_-^k \left( \zeta_b (\varepsilon), \varepsilon \right)  
= F_+^k \left( \zeta_b (\varepsilon), \varepsilon \right),
\quad  \text{for}\ \ W_\text{min} < \varepsilon < 0,
\label{HS-3-it-e}
\end{align}
\end{subequations}
where $F_\pm^k$ and $n^k$ are the $k$th iteration corresponding to
$F_\pm$ and $n$, respectively. This scheme, starting from the zero initial values, i.e.,
\begin{align}\label{HS-3-it-init}
F_\pm^0 = n^0 = 0,
\end{align}
will be used in the proofs in Sec.~\ref{sec:math}. In \cite{AGK22}, essentially the same
scheme with different initial values has been used in the numerical analysis of
the problem \eqref{HS-3}, as well as in the construction of a model
of the boundary condition for the Boltzmann equation.

From \eqref{HS-3-it}, $F_\pm^k$ can be solved in terms of $n^{k-1}$
as follows:
\begin{subequations}\label{HS-3-it-intg}
\begin{align}
& F_+^k (\zeta, \varepsilon) = \theta (\zeta_a(\varepsilon), \zeta; \varepsilon)
\Bigg( F_-^k (\zeta_a(\varepsilon), \varepsilon)
\nonumber \\
& \qquad \qquad \qquad \quad
+ \int_{\zeta_a(\varepsilon)}^\zeta \theta(s, \zeta_a(\varepsilon); \varepsilon)
\frac{n^{k-1}(s) \mathcal{M}(s, \varepsilon) \text{d}s}
{\tau(s) \sqrt{2 [\varepsilon - W(s)]}} \Bigg),
\label{HS-3-it-intg-a} \\
& F_-^k (\zeta, \varepsilon) = \bm{1}_{\varepsilon > 0}\,
\theta (\zeta, \infty; \varepsilon)
\Bigg( F_\infty (-\sqrt{2\varepsilon})
\nonumber \\
& \qquad \qquad \qquad \quad
+ \int_{\zeta}^\infty \theta(\infty, s; \varepsilon)
\frac{n^{k-1}(s) \mathcal{M}(s, \varepsilon) \text{d}s}
{\tau(s) \sqrt{2 [\varepsilon - W(s)]}} \Bigg)
\nonumber \\
& \qquad \qquad \;\;
+ \bm{1}_{W_\text{min} < \varepsilon<0}\,
\theta (\zeta, \zeta_b(\varepsilon); \varepsilon)
\Bigg( F_+^k (\zeta_b(\varepsilon), \varepsilon)
\nonumber \\
& \qquad \qquad \qquad
+ \int_{\zeta}^{\zeta_b(\varepsilon)} \theta(\zeta_b(\varepsilon), s; \varepsilon)
\frac{n^{k-1}(s) \mathcal{M}(s, \varepsilon) \text{d}s}
{\tau(s) \sqrt{2 [\varepsilon - W(s)]}} \Bigg),
\label{HS-3-it-intg-b}
\end{align}
\end{subequations}
where
\begin{align}\label{theta}
\theta(a, b; \varepsilon) = 
\exp\left(- \int_a^b \frac{\text{d} s}{\tau(s)\sqrt{2[\varepsilon-W(s)]}}\right),
\end{align}
and $\bm{1}_{\mathcal{S}}$ is the indicator function of the set $\mathcal{S}$, that is,
$\bm{1}_{\mathcal{S}} =1$ for $w \in \mathcal{S}$ and $\bm{1}_{\mathcal{S}} =0$
for $w \notin \mathcal{S}$ with $w$ being the relevant variable.
It should be noted that $F_-^k (\zeta, \varepsilon)$ and thus $F_+^k (\zeta, \varepsilon)$
are generally discontinuous at $\varepsilon=0$, i.e.,
$\lim_{\varepsilon \to 0^-} F_\pm^k (\zeta, \varepsilon) \ne \lim_{\varepsilon \to 0^+} F_\pm^k (\zeta, \varepsilon)$.
If necessary, the value of $F_\pm (\zeta, \varepsilon)$ at $\varepsilon=0$ may
naturally be defined by $\lim_{\varepsilon \to 0^+} F_\pm (\zeta, \varepsilon)$.
We should also note that
$\bm{1}_{\varepsilon > 0}$ and $\bm{1}_{W_\text{min} < \varepsilon<0}$
in \eqref{HS-3-it-intg} implicitly mean  
$\bm{1}_{\varepsilon > 0} \bm{1}_{\zeta > \zeta_a(\varepsilon)}$
and
$\bm{1}_{W_\text{min} < \varepsilon<0} \bm{1}_{\zeta_a(\varepsilon) < \zeta < \zeta_b(\varepsilon)}$,
respectively, since the range of $\zeta$ is given by \eqref{range-zeta}.
This convention will be used unless confusion arises.

One easily checks that
\begin{subequations}\label{theta-prop-1}
\begin{align}
& \theta (a, b; \varepsilon)\, \theta (b, c; \varepsilon) = \theta (a, c; \varepsilon),
\qquad
\theta (b, a; \varepsilon) = \theta(a, b; \varepsilon)^{-1},
\\
& a < b \implies 0 < \theta(a, b; \varepsilon) < 1, \qquad
a > b \implies \theta(a, b; \varepsilon) > 1.
\end{align}
\end{subequations}
On the other hand, defining 
\begin{align}\label{def-mu}
\mu(s,\varepsilon):=\frac{1}{\tau(s)\sqrt{2[\varepsilon-W(s)]}}>0,
\end{align}
one readily obtains
\begin{align}\label{diff-theta}
\partial_a \theta (a, b; \varepsilon) = \theta (a, b; \varepsilon) \mu(a, \varepsilon),
\qquad
\partial_b \theta (a, b; \varepsilon) = - \theta(a, b; \varepsilon) \mu(b, \varepsilon).
\end{align}
These properties of the function $\theta(a, b; \varepsilon)$ will be used
repeatedly in Sec.~\ref{sec:math}.

\section{Mathematical properties of half-space problem for the physisorbate layer}\label{sec:math}

In this section, we prove some mathematical properties of the half-space
problem \eqref{HS-2} for the physisorbate layer. The structure of the problem differs
from that of the traditional half-space problems of the linearized
Boltzmann equation relevant to Knudsen layers \cite{BCN86,GP89,BGS06}
in the following points:
\begin{itemize}
\item[(a)] the gas molecules are subject to an external attractive-repulsive potential;
\item[(b)] the gas molecules interact only with phonons;
\item[(c)] there are no gas molecules on the surface $\zeta = 0$
because of the infinite potential barrier there.
\end{itemize}

Before getting to the main points, we recall here that the potential $W(\zeta)$
and the relaxation time $\tau(\zeta)$ appearing in this section are, respectively,
assumed to satisfy the dimensionless version of the conditions \eqref{potential}
and (i)--(iv) in Sec.~\ref{subsec:kin-model} and that of the conditions
\eqref{tauph-limit}, (v), and (vi) there.
We note that the property (c), which has been used in deriving
\eqref{cons-c}, is a physical consequence from the assumption \eqref{potential}
for the potential. In this section, we consider the class of solutions to the
problem \eqref{HS-2} satisfying the property (c) (cf. Theorem \ref{thm-exist}
below).

\subsection{Main results}\label{subsec:main}

The main results are stated as follows:

\begin{theorem}\label{thm-exist}
Assume that $0 \le F_\infty (c_z) \le (A/\sqrt{2\pi})\exp(-c_z^2/2)$ for some
positive constant $A$. Then, the problem \eqref{HS-2} has the unique
solution satisfying the inequality
$0 \le F(\zeta, c_z) \le C (1/\sqrt{2\pi}) \exp \left( - c_z^2/2 - W(\zeta) \right)$
for some constant $C$. Moreover, the limit $\lim_{\zeta \to +\infty} F(\zeta, c_z)$
exists for $c_z > 0$.
\end{theorem}

\noindent
In fact, the resulting inequality for $F(\zeta, c_z)$ indicates that
$F(\zeta, c_z)$ enjoys the property (c) because
$\lim_{\zeta \to 0^+} \exp \left( - c_z^2/2 - W(\zeta) \right) = 0$.
The existence in Theorem \ref{thm-exist} is proved for the transformed problem
\eqref{HS-3} and \eqref{HS-3-continuity}, rather than the problem \eqref{HS-2},
with the help of the iteration scheme in Sec.~\ref{subsec:iteration}.
It is shown that the sequences $\{n^k\}$ and $\{F_\pm^k\}$, based on
\eqref{HS-3-it} and \eqref{HS-3-it-init}, converge
exponentially fast with respect to the number of iteration $k$.
Let
\begin{align*}
\ell = \int_0^\infty \frac{\text{d}s}{\tau(s)} < +\infty, \qquad
K = \frac{\sqrt{2} \zeta_\text{min}}{\tau(0)} + \frac{\ell}{\sqrt{2}}.
\end{align*}
Then, we have the following:

\begin{theorem}\label{thrm-exp-conv}
Under the same conditions as in Theorem \ref{thm-exist}, i.e.,
$0 \le F_\infty (-\sqrt{2 \varepsilon}) \le (A/\sqrt{2\pi})\exp(-\varepsilon)$,
the sequence $\{ n^k \}$ is nonnegative, nondecreasing in $k$, and converges
exponentially fast to its limit $n\, (\ge 0)$ as $k \to \infty$ in
$L^1 \left( \mathbf{R}^+; \mathrm{d}\zeta/\tau(\zeta) \right)$: 
\begin{align*}
\| n - n^k \|_{L^1 \left( \mathbf{R}^+; \mathrm{d} \zeta/\tau(\zeta) \right)}
= \int_0^\infty \vert n(\zeta) - n^k(\zeta) \vert \frac{\mathrm{d} \zeta}{\tau(\zeta)}
< A \ell e^{\vert W_{\mathrm{min}} \vert} \mathcal{L}^k,
\end{align*}
where
\begin{align*}
0 \le \mathcal{L} \le
1 - \frac{e^{-K}}{\sqrt{\pi}} \int_{\vert W_{\mathrm{min}} \vert + 1}^\infty
\frac{e^{-u}}{\sqrt{u}} \mathrm{d} u < 1.
\end{align*}
In addition, the sequence $\{ F_\pm^k \}$
is nonnegative, nondecreasing in $k$, and converges exponentially
fast to their limits $F_\pm\, (\ge 0)$ as $k \to \infty$ in
$
L^1 \left( \mathbf{R}^+ \times \mathbf{R}; \bm{1}_{\varepsilon > W(\zeta)}
\mathrm{d}\varepsilon \mathrm{d}\zeta 
/\tau(\zeta) \sqrt{2 [\varepsilon - W(\zeta)]} \right)
$:
\begin{align*}
& \| F_+ - F_+^k \|_{L^1 \left( \mathbf{R}^+ \times \mathbf{R}; \bm{1}_{\varepsilon > W(\zeta)}
\mathrm{d}\varepsilon \mathrm{d}\zeta
/\tau(\zeta) \sqrt{2 [\varepsilon - W(\zeta)]} \right) }
\\
& \qquad + \| F_- - F_-^k \|_{L^1 \left( \mathbf{R}^+ \times \mathbf{R}; \bm{1}_{\varepsilon > W(\zeta)}
\mathrm{d}\varepsilon \mathrm{d}\zeta
/\tau(\zeta) \sqrt{2 [\varepsilon - W(\zeta)]} \right)}
\\
& \quad = \int_0^\infty\!\! \int_{-\infty}^\infty
\big[ \vert F_+(\zeta,\varepsilon) - F_+^k(\zeta,\varepsilon) \vert
+ \vert F_-(\zeta,\varepsilon) - F_-^k(\zeta,\varepsilon) \vert \big]
\frac{\bm{1}_{\varepsilon > W(\zeta)} \mathrm{d}\varepsilon \mathrm{d}\zeta}
{\tau(\zeta) \sqrt{2 [\varepsilon - W(\zeta)]}}
\\
& \quad 
< A \ell e^{\vert W_{\mathrm{min}} \vert} \mathcal{L}^k.
\end{align*}
\end{theorem}

\subsection{Proof of uniqueness}\label{subsec:uniqueness}

Here, we prove the uniqueness in Theorem \ref{thm-exist}.

\begin{lemma}\label{lem-unique}
If the problem \eqref{HS-2} has a solution $F(\zeta, c_z)$ satisfying
$\vert F(\zeta, c_z) \vert \le C e^{-W(\zeta)} \mathcal{M}(c_z)$ for some
constant $C$, then the solution is unique.
\end{lemma}

\begin{proof}
Let us denote $\partial_\zeta = \partial/\partial \zeta$,
$\partial_{c_z} = \partial/\partial c_z$, $G'(\zeta) = \text{d} G(\zeta)/\text{d} \zeta$,
where $G(\zeta)$ is an arbitrary function of $\zeta$, and
\[
E(\zeta, c_z) = e^{-W(\zeta)} \mathcal{M}(c_z).
\]
This is indeed a modified Maxwellian that is a natural equilibrium solution
to the layer equation \eqref{HS-2-a}.
By linearity of the problem \eqref{HS-2}, it suffices to consider the problem
with $F_\infty(c_z)=0$ and to prove that the only possible solution satisfying the condition
$\vert F(\zeta, c_z) \vert \le C E(\zeta, c_z)$ is $F(\zeta, c_z)=0$.

\indent
If we multiply both sides of \eqref{HS-2-a} by $F(\zeta, c_z)/E(\zeta, c_z)$ and
take account of the relations $\partial_\zeta [1/E(\zeta, c_z)]=W'(\zeta)/E(\zeta, c_z)$
and $\partial_{c_z} [1/E(\zeta, c_z)]=c_z/E(\zeta, c_z)$, we have
\begin{align}
& \partial_\zeta \left( \frac{1}{2} c_z \frac{F(\zeta, c_z)^2}{E(\zeta, c_z)} \right)
- \partial_{c_z} \left( \frac{1}{2} W'(\zeta) \frac{F(\zeta, c_z)^2}{E(\zeta, c_z)} \right)
\nonumber \\
& \qquad \qquad \qquad 
= \frac{e^{W(\zeta)}}{\tau(\zeta)} \left( n(\zeta) F(\zeta, c_z)
- \frac{F(\zeta, c_z)^2}{\mathcal{M}(c_z)} \right).
\end{align}
Integrating this equation with respect to $\zeta$ and $c_z$ over $(0, +\infty)$
and $(-\infty, +\infty)$, respectively, leads to
\begin{align}\label{HS-2-a-integ}
& \left[ \int_{-\infty}^{\infty} \frac{1}{2} c_z \frac{F(\zeta, c_z)^2}{E(\zeta, c_z)} \text{d}c_z
\right]_{\zeta=0}^{\zeta=+\infty}
- \left[ \int_0^{\infty} \frac{1}{2} W'(\zeta) \frac{F(\zeta, c_z)^2}{E(\zeta, c_z)} \text{d}\zeta
\right]_{c_z=-\infty}^{c_z=+\infty}
\nonumber \\
& \qquad
= \int_0^\infty \frac{e^{W(\zeta)}}{\tau(\zeta)} \left( n(\zeta)^2
- \int_{-\infty}^\infty \frac{F(\zeta, c_z)^2}{\mathcal{M}(c_z)} \text{d} c_z
\right) \text{d}\zeta.
\end{align}
Because $0 \le F(\zeta, c_z)^2/E(\zeta, c_z) < C^2 e^{-W(\zeta)} \mathcal{M}(c_z)$,
we find
\begin{align*}
\left\vert \int_{-\infty}^{\infty} \frac{1}{2} c_z \frac{F(\zeta, c_z)^2}{E(\zeta, c_z)} \text{d}c_z
\right\vert \le C^2 e^{-W(\zeta)} \int_{-\infty}^\infty \frac{1}{2} \vert c_z \vert
\mathcal{M}(c_z) \text{d}c_z \to 0,
\end{align*}
as $\zeta \to 0^+$ and
\begin{align*}
\left\vert \int_0^{\infty} \frac{1}{2} W'(\zeta) \frac{F(\zeta, c_z)^2}{E(\zeta, c_z)} \text{d}\zeta
\right\vert \le C^2 \mathcal{M}(c_z) \int_0^\infty \frac{1}{2} \vert W'(\zeta) \vert e^{-W(\zeta)} \text{d}\zeta
\to 0,
\end{align*}
as $\vert c_z \vert \to +\infty$. In addition, since $F_\infty (c_z)=0$,
it follows that
\begin{align*}
\lim_{\zeta \to +\infty}
\int_{-\infty}^{\infty} \frac{1}{2} c_z \frac{F(\zeta, c_z)^2}{E(\zeta, c_z)} \text{d}c_z
= \lim_{\zeta \to +\infty} 
\int_0^{\infty} \frac{1}{2} c_z \frac{F(\zeta, c_z)^2}{E(\zeta, c_z)} \text{d}c_z \ge 0.
\end{align*}
Therefore, with the help of \eqref{HS-2-b} and $\int_{-\infty}^\infty \mathcal{M}(c_z) \text{d} c_z =1$,
\eqref{HS-2-a-integ} gives
\begin{align*}
& \mathcal{I} := \int_0^\infty \frac{e^{W(\zeta)}}{\tau(\zeta)} \left[ 
\Big( \int_{-\infty}^\infty F(\zeta, c_z) \text{d}c_z \Big)^2 \right.
\nonumber \\
& \qquad \qquad \qquad \quad \left.
- \int_{-\infty}^\infty \mathcal{M}(c_z) \text{d} c_z
\int_{-\infty}^\infty \frac{F(\zeta, c_z)^2}{\mathcal{M}(c_z)} \text{d} c_z
\right] \text{d}\zeta \ge 0.
\end{align*}
On the other hand, it follows from the Cauchy--Schwarz inequality that
\begin{align*}
\mathcal{I} \le 0.
\end{align*}
This means that we are in the equality case in the Cauchy--Schwarz inequality
for a.e.~$\zeta > 0$, so that $F(\zeta, \cdot)$ must be proportional to
$\mathcal{M}$:
\begin{align*}
F(\zeta, c_z) = c(\zeta) \mathcal{M}(c_z), \;\;\;\;
\text{so that} \;\;\; c(\zeta) = n(\zeta) = \int_{-\infty}^\infty F(\zeta, c_z) \text{d}c_z.
\end{align*}
The substitution of this form into \eqref{HS-2-a} leads to
\begin{align*}
c_z \mathcal{M}(c_z) [n'(\zeta) + W'(\zeta) n(\zeta)] = 0,
\end{align*}
which gives
\begin{align*}
n(\zeta) = B e^{-W(\zeta)},
\end{align*}
for some constant $B$. Therefore, letting $\zeta \to +\infty$, one finds that
\begin{align*}
F(\zeta, c_z) = B e^{-W(\zeta)} \mathcal{M}(c_z) \to B \mathcal{M}(c_z),
\;\;\; \text{as}\;\; \zeta \to +\infty,
\end{align*}
since $W(\zeta) \to 0$ as $\zeta \to +\infty$.  Then, the boundary condition
\eqref{HS-2-d} with $F_\infty (c_z)=0$ implies that $B=0$, so that we have
$F(\zeta, c_z) = 0$.
\end{proof}

\subsection{Proof of existence}\label{subsec:existence}

In this section, we prove the existence of the solution and its limit as
$\zeta \to \infty$ in Theorem \ref{thm-exist}. The existence proof is based on the
approximating sequence $\{ F_\pm^k \}$ constructed by
\eqref{HS-3-it-intg} with \eqref{HS-3-it-init}.

\subsubsection{Computing $F^k_\pm(\zeta_a(\varepsilon), \varepsilon)$ and
$F^k_\pm(\zeta_b(\varepsilon), \varepsilon)$}

In the case where $\varepsilon > 0$, one obtains from \eqref{HS-3-it-intg-b}
and \eqref{HS-3-it-d}
\begin{align}\label{Fzataa-eps>0}
& F^k_\pm(\zeta_a(\varepsilon),\varepsilon)
\nonumber \\
& \;\;\; = 
\theta (\zeta_a(\varepsilon),\infty;\varepsilon) F_\infty (-\sqrt{2\varepsilon})
+\int_{\zeta_a(\varepsilon)}^\infty \theta(\zeta_a(\varepsilon),s;\varepsilon)
\frac{n^{k-1}(s) \mathcal{M}(s,\varepsilon) \text{d}s}{\tau(s)\sqrt{2[\varepsilon-W(s)]}}
\nonumber \\
& \;\;\;
= \theta(\zeta_a(\varepsilon),\infty;\varepsilon) F_\infty (-\sqrt{2\varepsilon} )
+\int_{\zeta_a(\varepsilon)}^\infty \theta(\zeta_a(\varepsilon),s;\varepsilon)\mu(s,\varepsilon)
n^{k-1}(s)\mathcal{M}(s,\varepsilon) \text{d}s
\nonumber \\
& \;\;\;
= \theta(\zeta_a(\varepsilon),\infty;\varepsilon) F_\infty\left(-\sqrt{2\varepsilon}\right)
- \int_{\zeta_a(\varepsilon)}^\infty \partial_s\theta(\zeta_a(\varepsilon),s;\varepsilon)
n^{k-1}(s)\mathcal{M}(s,\varepsilon)\text{d}s.
\end{align}
On the other hand, if $W_{\text{min}}<\varepsilon<0$, we first use \eqref{HS-3-it-intg-b}
to compute
\begin{align*}
F^k_-(\zeta_a(\varepsilon),\varepsilon)
& = \theta(\zeta_a(\varepsilon),\zeta_b(\varepsilon);\varepsilon)
\Bigg(F^k_+(\zeta_b(\varepsilon),\varepsilon)
\nonumber \\
& \qquad \quad
+\int_{\zeta_a(\varepsilon)}^{\zeta_b(\varepsilon)}\theta(\zeta_b(\varepsilon),s;\varepsilon)
\frac{n^{k-1}(s)\mathcal{M}(s,\varepsilon)\text{d}s}{\tau(s)\sqrt{2[\varepsilon-W(s)]}}\Bigg),
\end{align*}
and then inject in the r.h.s. the following expression of $F^k_+(\zeta_b(\varepsilon),\varepsilon)$,
obtained from \eqref{HS-3-it-intg-a}:
\begin{align*}
F_+^k(\zeta_b(\varepsilon),\varepsilon)
& =\theta(\zeta_a(\varepsilon),\zeta_b(\varepsilon);\varepsilon)
\Bigg(F_-^k(\zeta_a(\varepsilon),\varepsilon)
\nonumber \\
& \qquad \quad
+\int_{\zeta_a(\varepsilon)}^{\zeta_b(\varepsilon)}\theta(s,\zeta_a(\varepsilon);\varepsilon)
\frac{n^{k-1}(s)\mathcal{M}(s,\varepsilon)\text{d}s}{\tau(s)\sqrt{2[\varepsilon-W(s)]}}\Bigg).
\end{align*}
This leads us to the equality
\begin{align}\label{F-zetaa}
& F^k_\pm(\zeta_a(\varepsilon),\varepsilon) [1-\theta(\zeta_a(\varepsilon),\zeta_b(\varepsilon);\varepsilon)^2]
\nonumber \\
& \qquad
=\theta(\zeta_a(\varepsilon),\zeta_b(\varepsilon);\varepsilon)
\int_{\zeta_a(\varepsilon)}^{\zeta_b(\varepsilon)}\theta(\zeta_b(\varepsilon),s;\varepsilon)
\mu(s, \varepsilon) n^{k-1}(s)\mathcal{M}(s,\varepsilon)\text{d}s
\nonumber \\
& \qquad \;\;
+\theta(\zeta_a(\varepsilon),\zeta_b(\varepsilon);\varepsilon)^2
\int_{\zeta_a(\varepsilon)}^{\zeta_b(\varepsilon)}\theta(s,\zeta_a(\varepsilon);\varepsilon)
\mu(s, \varepsilon) n^{k-1}(s)\mathcal{M}(s,\varepsilon)\text{d}s,
\end{align}
with the help of \eqref{HS-3-it-d} and \eqref{def-mu}. Likewise
\begin{align}\label{F-zetab}
& F^k_\pm(\zeta_b(\varepsilon),\varepsilon) [1-\theta(\zeta_a(\varepsilon),\zeta_b(\varepsilon);\varepsilon)^2]
\nonumber \\
& \qquad
=\theta(\zeta_a(\varepsilon),\zeta_b(\varepsilon);\varepsilon)
\int_{\zeta_a(\varepsilon)}^{\zeta_b(\varepsilon)}\theta(s,\zeta_a(\varepsilon);\varepsilon)
\mu(s, \varepsilon) n^{k-1}(s)\mathcal{M}(s,\varepsilon)\text{d}s
\nonumber \\
& \qquad \;\;
+\theta(\zeta_a(\varepsilon),\zeta_b(\varepsilon);\varepsilon)^2
\int_{\zeta_a(\varepsilon)}^{\zeta_b(\varepsilon)}\theta(\zeta_b(\varepsilon),s;\varepsilon)
\mu(s, \varepsilon) n^{k-1}(s)\mathcal{M}(s,\varepsilon)\text{d}s,
\end{align}
with the help of \eqref{HS-3-it-e} and \eqref{def-mu}.

Equation \eqref{F-zetaa} can be recast as follows: for $W_\text{min} < \varepsilon < 0$,
\begin{align*}
& F^k_\pm(\zeta_a(\varepsilon),\varepsilon)\, \theta(\zeta_a(\varepsilon),\zeta_b(\varepsilon);\varepsilon)\,
[ \theta(\zeta_b(\varepsilon),\zeta_a(\varepsilon);\varepsilon)-\theta(\zeta_a(\varepsilon),\zeta_b(\varepsilon);\varepsilon) ]
\nonumber \\
& \;\; =\theta(\zeta_a(\varepsilon),\zeta_b(\varepsilon);\varepsilon)
\int_{\zeta_a(\varepsilon)}^{\zeta_b(\varepsilon)} [-\partial_s\theta(\zeta_b(\varepsilon),s;\varepsilon)
\nonumber \\
& \qquad \qquad \qquad \qquad \qquad \qquad
+\partial_s\theta(s,\zeta_b(\varepsilon);\varepsilon)] n^{k-1}(s)\mathcal{M}(s,\varepsilon)\text{d}s,
\end{align*}
or, equivalently,
\begin{align}\label{Fzetaa-final}
F^k_\pm(\zeta_a(\varepsilon),\varepsilon)
= \int_{\zeta_a(\varepsilon)}^{\zeta_b(\varepsilon)} \tfrac{\partial_s\theta(s,\zeta_b(\varepsilon);\varepsilon)-\partial_s\theta(\zeta_b(\varepsilon),s;\varepsilon)}
{\theta(\zeta_b(\varepsilon),\zeta_a(\varepsilon);\varepsilon)-\theta(\zeta_a(\varepsilon),\zeta_b(\varepsilon);\varepsilon)}
n^{k-1}(s)\mathcal{M}(s,\varepsilon)\text{d}s.
\end{align}
Similarly, \eqref{F-zetab} leads to the following expression: for  $W_\text{min} < \varepsilon < 0$,
\begin{align}\label{Fzetab-final}
F^k_\pm(\zeta_b(\varepsilon),\varepsilon)
= \int_{\zeta_a(\varepsilon)}^{\zeta_b(\varepsilon)} \tfrac{\partial_s\theta(s,\zeta_a(\varepsilon);\varepsilon)-\partial_s\theta(\zeta_a(\varepsilon),s;\varepsilon)}
{\theta(\zeta_b(\varepsilon),\zeta_a(\varepsilon);\varepsilon)-\theta(\zeta_a(\varepsilon),\zeta_b(\varepsilon);\varepsilon)}
n^{k-1}(s)\mathcal{M}(s,\varepsilon)\text{d}s.
\end{align}
One easily checks that
\begin{align*}
& \tfrac{\partial_s\theta(s,\zeta_b(\varepsilon);\varepsilon)-\partial_s\theta(\zeta_b(\varepsilon),s;\varepsilon)}
{\theta(\zeta_b(\varepsilon),\zeta_a(\varepsilon);\varepsilon)-\theta(\zeta_a(\varepsilon),\zeta_b(\varepsilon);\varepsilon)}\ge 0,
\quad
\tfrac{\partial_s\theta(s,\zeta_a(\varepsilon);\varepsilon)-\partial_s\theta(\zeta_a(\varepsilon),s;\varepsilon)}
{\theta(\zeta_b(\varepsilon),\zeta_a(\varepsilon);\varepsilon)-\theta(\zeta_a(\varepsilon),\zeta_b(\varepsilon);\varepsilon)}\ge 0,
\end{align*}
and
\begin{align*}
& \int_{\zeta_a(\varepsilon)}^{\zeta_b(\varepsilon)}
\tfrac{\partial_s\theta(s,\zeta_b(\varepsilon);\varepsilon)-\partial_s\theta(\zeta_b(\varepsilon),s;\varepsilon)}
{\theta(\zeta_b(\varepsilon),\zeta_a(\varepsilon);\varepsilon)-\theta(\zeta_a(\varepsilon),\zeta_b(\varepsilon);\varepsilon)}\text{d}s=1,
\\
& \int_{\zeta_a(\varepsilon)}^{\zeta_b(\varepsilon)}
\tfrac{\partial_s\theta(s,\zeta_a(\varepsilon);\varepsilon)-\partial_s\theta(\zeta_a(\varepsilon),s;\varepsilon)}
{\theta(\zeta_b(\varepsilon),\zeta_a(\varepsilon);\varepsilon)-\theta(\zeta_a(\varepsilon),\zeta_b(\varepsilon);\varepsilon)}\text{d}s=1.
\end{align*}
Therefore $F^k_\pm(\zeta_a(\varepsilon),\varepsilon)$ and $F^k_\pm(\zeta_b(\varepsilon),\varepsilon)$
are averages of $n^{k-1}(s)\mathcal{M}(s,\varepsilon)$ for $s\in[\zeta_a(\varepsilon),\zeta_b(\varepsilon)]$.

\subsubsection{Simplifying the formulas \eqref{HS-3-it-intg}}

Next we plug \eqref{Fzataa-eps>0}, \eqref{Fzetaa-final}, and \eqref{Fzetab-final}
in \eqref{HS-3-it-intg}.
Then, \eqref{HS-3-it-intg-a} leads to
\begin{align*}
F_+^k(\zeta,\varepsilon) &= \theta(\zeta_a(\varepsilon),\zeta;\varepsilon)
F_-^k(\zeta_a(\varepsilon),\varepsilon)+\int_{\zeta_a(\varepsilon)}^\zeta
\theta(s,\zeta;\varepsilon)\mu(s,\varepsilon)n^{k-1}(s)\mathcal{M}(s,\varepsilon)\text{d}s
\\
& = \theta(\zeta_a(\varepsilon),\zeta;\varepsilon)
F_-^k(\zeta_a(\varepsilon),\varepsilon)+\int_{\zeta_a(\varepsilon)}^\zeta
\partial_s\theta(s,\zeta;\varepsilon)n^{k-1}(s)\mathcal{M}(s,\varepsilon)\text{d}s
\\
& = \bm{1}_{\varepsilon > 0}\, \theta(\zeta_a(\varepsilon),\zeta;\varepsilon)\,
\theta(\zeta_a(\varepsilon),\infty;\varepsilon) F_\infty (-\sqrt{2\varepsilon} )
\\
& \quad -\bm{1}_{\varepsilon > 0}\, \theta(\zeta_a(\varepsilon),\zeta;\varepsilon)
\int_{\zeta_a(\varepsilon)}^\infty \partial_s\theta(\zeta_a(\varepsilon),s;\varepsilon)
n^{k-1}(s)\mathcal{M}(s,\varepsilon)\text{d}s
\\
& \quad +\bm{1}_{W_\text{min}<\varepsilon<0}\, \theta(\zeta_a(\varepsilon),\zeta;\varepsilon)
\\
& \qquad \qquad \times
\int_{\zeta_a(\varepsilon)}^{\zeta_b(\varepsilon)}
\tfrac{\partial_s\theta(s,\zeta_b(\varepsilon);\varepsilon)-\partial_s\theta(\zeta_b(\varepsilon),s;\varepsilon)}
{\theta(\zeta_b(\varepsilon),\zeta_a(\varepsilon);\varepsilon)-\theta(\zeta_a(\varepsilon),\zeta_b(\varepsilon);\varepsilon)}
n^{k-1}(s)\mathcal{M}(s,\varepsilon)\text{d}s
\\
& \quad +\bm{1}_{W(\zeta)<\varepsilon}\int_{\zeta_a(\varepsilon)}^\zeta
\partial_s\theta(s,\zeta;\varepsilon)n^{k-1}(s)\mathcal{M}(s,\varepsilon)\text{d}s,
\end{align*}
while \eqref{HS-3-it-intg-b} is recast as
\begin{align*}
& F^k_-(\zeta,\varepsilon)
\\
& \;\; =\bm{1}_{\varepsilon > 0}\, \theta(\zeta,\infty;\varepsilon) \Bigg( F_\infty (-\sqrt{2\varepsilon} )
+\int_\zeta^\infty\theta(\infty,s;\varepsilon)\tfrac{n^{k-1}(s)\mathcal{M}(s,\varepsilon)\text{d}s}
{\tau(s)\sqrt{2[\varepsilon-W(s)]}}\Bigg)
\\
& \;\;\quad +\bm{1}_{W_\text{min}<\varepsilon<0}\, \theta(\zeta,\zeta_b(\varepsilon);\varepsilon)
\Bigg( F^k_+(\zeta_b(\varepsilon),\varepsilon)
\\
& \qquad \qquad \qquad \qquad \qquad \qquad \qquad \qquad
+ \int_\zeta^{\zeta_b(\varepsilon)}
\theta(\zeta_b(\varepsilon),s;\varepsilon)\tfrac{n^{k-1}(s)\mathcal{M}(s,\varepsilon)\text{d}s}
{\tau(s)\sqrt{2[\varepsilon-W(s)]}}\Bigg)
\\
& \;\; =\bm{1}_{\varepsilon > 0}\, \theta(\zeta,\infty;\varepsilon) F_\infty (-\sqrt{2\varepsilon} )
+\bm{1}_{\varepsilon > 0}\int_\zeta^\infty\theta(\zeta,s;\varepsilon)\mu(s,\varepsilon)
n^{k-1}(s)\mathcal{M}(s,\varepsilon)\text{d}s
\\
& \;\;\quad +\bm{1}_{W_\text{min}<\varepsilon<0}\, \theta(\zeta,\zeta_b(\varepsilon);\varepsilon)
\\
& \qquad \qquad \qquad \qquad \times
\int_{\zeta_a(\varepsilon)}^{\zeta_b(\varepsilon)}
\tfrac{\partial_s\theta(s,\zeta_a(\varepsilon);\varepsilon)-\partial_s\theta(\zeta_a(\varepsilon),s;\varepsilon)}
{\theta(\zeta_b(\varepsilon),\zeta_a(\varepsilon);\varepsilon)-\theta(\zeta_a(\varepsilon),\zeta_b(\varepsilon);\varepsilon)}
n^{k-1}(s)\mathcal{M}(s,\varepsilon)\text{d}s
\\
& \;\;\quad +\bm{1}_{W_\text{min}<\varepsilon<0}\int_\zeta^{\zeta_b(\varepsilon)}
\theta(\zeta,s;\varepsilon)\mu(s,\varepsilon)n^{k-1}(s)\mathcal{M}(s,\varepsilon)\text{d}s
\\
& \;\; =\bm{1}_{\varepsilon > 0}\, \theta(\zeta,\infty;\varepsilon) F_\infty (-\sqrt{2\varepsilon} )
-\bm{1}_{\varepsilon > 0}\int_\zeta^\infty \partial_s\theta(\zeta,s;\varepsilon)
n^{k-1}(s)\mathcal{M}(s,\varepsilon)\text{d}s
\\
& \;\;\quad +\bm{1}_{W_\text{min}<\varepsilon<0}\, \theta(\zeta,\zeta_b(\varepsilon);\varepsilon)
\\
& \qquad \qquad \qquad \qquad \times
\int_{\zeta_a(\varepsilon)}^{\zeta_b(\varepsilon)}
\tfrac{\partial_s\theta(s,\zeta_a(\varepsilon);\varepsilon)-\partial_s\theta(\zeta_a(\varepsilon),s;\varepsilon)}
{\theta(\zeta_b(\varepsilon),\zeta_a(\varepsilon);\varepsilon)-\theta(\zeta_a(\varepsilon),\zeta_b(\varepsilon);\varepsilon)}
n^{k-1}(s)\mathcal{M}(s,\varepsilon)\text{d}s
\\
& \;\;\quad -\bm{1}_{W_\text{min}<\varepsilon<0}\int_\zeta^{\zeta_b(\varepsilon)} \partial_s\theta(\zeta,s;\varepsilon)
n^{k-1}(s)\mathcal{M}(s,\varepsilon)\text{d}s.
\end{align*}
Summarizing, we have proved that
\begin{subequations}\label{Fpm}
\begin{align}
& F_+^k(\zeta,\varepsilon) = \bm{1}_{\varepsilon > 0}\, \theta(\zeta_a(\varepsilon),\zeta;\varepsilon)\,
\theta(\zeta_a(\varepsilon),\infty;\varepsilon) F_\infty (-\sqrt{2\varepsilon} )
\nonumber \\
& \qquad \qquad \quad +\int_0^\infty K_+(\zeta,s,\varepsilon)n^{k-1}(s)\mathcal{M}(s,\varepsilon)\text{d}s,
\label{Fpm-a} \\
& F_-^k(\zeta,\varepsilon) = \bm{1}_{\varepsilon > 0}\, \theta(\zeta,\infty;\varepsilon)
F_\infty (-\sqrt{2\varepsilon} )
\nonumber \\
& \qquad \qquad \quad +\int_0^\infty K_-(\zeta,s,\varepsilon)n^{k-1}(s)\mathcal{M}(s,\varepsilon)\text{d}s,
\label{Fpm-b}
\end{align}
\end{subequations}
where
\begin{subequations}\label{Kpm}
\begin{align}
& K_+(\zeta,s,\varepsilon):= -\bm{1}_{\varepsilon > 0}\bm{1}_{\zeta>\zeta_a(\varepsilon)}\,
\theta(\zeta_a(\varepsilon),\zeta,\varepsilon)\, \partial_s\theta(\zeta_a(\varepsilon),s,\varepsilon)\bm{1}_{s>\zeta_a(\varepsilon)}
\nonumber \\
& \qquad \qquad \qquad
+\bm{1}_{W_\text{min}<\varepsilon<0}\bm{1}_{\zeta_a(\varepsilon)<\zeta<\zeta_b(\varepsilon)}\,
\theta(\zeta_a(\varepsilon),\zeta,\varepsilon)
\nonumber \\
& \qquad \qquad \qquad \qquad \qquad \times
\tfrac{\partial_s\theta(s,\zeta_b(\varepsilon);\varepsilon)-\partial_s\theta(\zeta_b(\varepsilon),s;\varepsilon)}
{\theta(\zeta_b(\varepsilon),\zeta_a(\varepsilon);\varepsilon)-\theta(\zeta_a(\varepsilon),\zeta_b(\varepsilon);\varepsilon)}
\bm{1}_{\zeta_a(\varepsilon)<s<\zeta_b(\varepsilon)}
\nonumber \\
& \qquad \qquad \qquad
+\bm{1}_{\varepsilon > 0}\partial_s\theta(s,\zeta,\varepsilon)\bm{1}_{\zeta_a(\varepsilon)<s<\zeta}
\nonumber \\
& \qquad \qquad \qquad
+\bm{1}_{W_\text{min}<\varepsilon<0}\partial_s\theta(s,\zeta,\varepsilon)
\bm{1}_{\zeta_a(\varepsilon)<s<\zeta<\zeta_b(\varepsilon)},
\label{Kpm-a} \\[2mm]
& K_-(\zeta,s,\varepsilon):=-\bm{1}_{\varepsilon > 0} \partial_s\theta(\zeta,s;\varepsilon)\bm{1}_{s>\zeta>\zeta_a(\varepsilon)}
\nonumber \\
& \qquad \qquad \qquad
-\bm{1}_{W_\text{min}<\varepsilon<0}\, \partial_s\theta(\zeta,s;\varepsilon)\bm{1}_{\zeta_a(\varepsilon)<\zeta<s<\zeta_b(\varepsilon)}
\nonumber \\
& \qquad \qquad \qquad
+\bm{1}_{W_\text{min}<\varepsilon<0}\bm{1}_{\zeta_a(\varepsilon)<\zeta<\zeta_b(\varepsilon)}\theta(\zeta,\zeta_b(\varepsilon);\varepsilon)
\nonumber \\
& \qquad \qquad \qquad \qquad \qquad \times
\tfrac{\partial_s\theta(s,\zeta_a(\varepsilon);\varepsilon)-\partial_s\theta(\zeta_a(\varepsilon),s;\varepsilon)}{\theta(\zeta_b(\varepsilon),\zeta_a(\varepsilon);\varepsilon)-\theta(\zeta_a(\varepsilon),\zeta_b(\varepsilon);\varepsilon)}\bm{1}_{\zeta_a(\varepsilon)<s<\zeta_b(\varepsilon)}.
\label{Kpm-b}
\end{align}
\end{subequations}
Since both position variables $\zeta$ and $s$ appear in \eqref{Kpm}, the indicator functions
for these variables, such as $\bm{1}_{\zeta>\zeta_a(\varepsilon)}$, $\bm{1}_{s>\zeta_a(\varepsilon)}$,
and $\bm{1}_{\zeta_a(\varepsilon)<\zeta<s<\zeta_b(\varepsilon)}$, are shown explicitly
to avoid confusion.

\subsubsection{Properties of $K_\pm(\zeta,s,\varepsilon)$}

Equations \eqref{def-mu} and \eqref{diff-theta} show that $-\partial_s \theta(\zeta, s; \varepsilon) > 0$
and $\partial_s \theta (s,\, \zeta; \varepsilon) > 0$, so that
\begin{align*}
K_\pm (\zeta, s, \varepsilon) \ge 0.
\end{align*}

On the other hand
\begin{align*}
& \int_0^\infty K_+(\zeta,s,\varepsilon)\text{d}s
\\
& \qquad =\bm{1}_{\varepsilon>0}\bm{1}_{\zeta>\zeta_a(\varepsilon)}\theta(\zeta_a(\varepsilon),\zeta; \varepsilon)
[1-\theta(\zeta_a(\varepsilon),\infty; \varepsilon)]
\\
& \qquad \quad +\bm{1}_{W_\text{min}<\varepsilon<0}\bm{1}_{\zeta_a(\varepsilon)<\zeta<\zeta_b(\varepsilon)}\,
\theta(\zeta_a(\varepsilon),\zeta; \varepsilon)
\\
& \qquad \quad + (\bm{1}_{W_\text{min}<\varepsilon<0}\bm{1}_{\zeta_a(\varepsilon)<\zeta<\zeta_b(\varepsilon)}
+\bm{1}_{\varepsilon>0}\bm{1}_{\zeta>\zeta_a(\varepsilon)}) [1-\theta(\zeta_a(\varepsilon),\zeta; \varepsilon)],
\end{align*}
or, equivalently,
\begin{align}\label{est-intK+}
& \int_0^\infty K_+(\zeta,s,\varepsilon)\text{d}s
\nonumber \\
& \qquad 
=(\bm{1}_{W_\text{min}<\varepsilon<0}\bm{1}_{\zeta_a(\varepsilon)<\zeta<\zeta_b(\varepsilon)}
+\bm{1}_{\varepsilon>0}\bm{1}_{\zeta>\zeta_a(\varepsilon)})\, \theta(\zeta_a(\varepsilon),\zeta;\varepsilon)
\nonumber \\
& \qquad \quad +(\bm{1}_{W_\text{min}<\varepsilon<0}\bm{1}_{\zeta_a(\varepsilon)<\zeta<\zeta_b(\varepsilon)}
+\bm{1}_{\varepsilon>0}\bm{1}_{\zeta>\zeta_a(\varepsilon)}) [1-\theta(\zeta_a(\varepsilon),\zeta;\varepsilon)]
\nonumber \\
& \qquad \quad -\bm{1}_{\varepsilon>0}\bm{1}_{\zeta>\zeta_a(\varepsilon)}\,
\theta(\zeta_a(\varepsilon),\zeta;\varepsilon)\, \theta(\zeta_a(\varepsilon),\infty;\varepsilon)
\nonumber \\
& \qquad =\bm{1}_{W(\zeta)<\varepsilon}-\bm{1}_{\varepsilon>0}\bm{1}_{\zeta>\zeta_a(\varepsilon)}\,
\theta(\zeta_a(\varepsilon),\zeta;\varepsilon)\, \theta(\zeta_a(\varepsilon),\infty;\varepsilon).
\end{align}
Likewise, one easily finds that
\begin{align}\label{est-intK-}
& \int_0^\infty K_-(\zeta,s,\varepsilon)\text{d}s
\nonumber \\
& \qquad 
=\bm{1}_{\varepsilon>0}\bm{1}_{\zeta>\zeta_a(\varepsilon)} [1-\theta(\zeta,\infty;\varepsilon)]
\nonumber \\
& \qquad \quad
+\bm{1}_{W_\text{min}<\varepsilon<0}\bm{1}_{\zeta_a(\varepsilon)<\zeta<\zeta_b(\varepsilon)}
[1-\theta(\zeta,\zeta_b(\varepsilon);\varepsilon)]
\nonumber \\
& \qquad \quad
+\bm{1}_{W_\text{min}<\varepsilon<0}\bm{1}_{\zeta_a(\varepsilon)<\zeta<\zeta_b(\varepsilon)}\,
\theta(\zeta,\zeta_b(\varepsilon);\varepsilon)
\nonumber \\
& \qquad
=\bm{1}_{\varepsilon>0}\bm{1}_{\zeta>\zeta_a(\varepsilon)}+\bm{1}_{W_\text{min}<\varepsilon<0}\bm{1}_{\zeta_a(\varepsilon)<\zeta<\zeta_b(\varepsilon)}-\bm{1}_{\varepsilon>0}\bm{1}_{\zeta>\zeta_a(\varepsilon)}\theta(\zeta,\infty;\varepsilon)
\nonumber \\
& \qquad
=\bm{1}_{W(\zeta)<\varepsilon}-\bm{1}_{\varepsilon>0}\bm{1}_{\zeta>\zeta_a(\varepsilon)}\theta(\zeta,\infty;\varepsilon).
\end{align}
\subsubsection{Monotonicity of the approximating sequence}

In this section, we seek to prove the following result for the
approximating sequences $\{F_\pm^k\}$ and $\{n^k\}$ based on
\eqref{HS-3-it} and \eqref{HS-3-it-init}.

\begin{lemma}\label{lem-monoton}
Assume that $0\le F_\infty (-\sqrt{2\varepsilon} ) \le Ae^{-\varepsilon}/\sqrt{2\pi}$. Then
\[
0=F^0_\pm(\zeta,\varepsilon) \le F^1_\pm(\zeta,\varepsilon) \le F^2_\pm(\zeta,\varepsilon)
\le \ldots \le F^k_\pm(\zeta,\varepsilon) \le \ldots \le Ae^{-\varepsilon}/\sqrt{2\pi}.
\]
Therefore
\[
0=n^0(\zeta) \le n^1(\zeta) \le n^2(\zeta) \le \ldots \le n^k(\zeta) \le \ldots \le Ae^{-W(\zeta)}.
\]
\end{lemma}

\begin{proof}
Since
\begin{align*}
& F_+^{k+1}(\zeta,\varepsilon)-F_+^k(\zeta,\varepsilon)=\int_0^\infty
K_+(\zeta,s,\varepsilon)[n^k(s)-n^{k-1}(s)]\mathcal{M}(s,\varepsilon)\text{d}s,
\\
& F_-^{k+1}(\zeta,\varepsilon)-F_-^k(\zeta,\varepsilon)=\int_0^\infty
K_-(\zeta,s,\varepsilon)[n^k(s)-n^{k-1}(s)]\mathcal{M}(s,\varepsilon)\text{d}s,
\end{align*}
with $K_\pm(\zeta,s,\varepsilon)\ge 0$, one has
\begin{align*}
n^k \ge n^{k-1} \implies F_\pm^{k+1} \ge F_\pm^k.
\end{align*}
On the other hand, by definition of $n^k(\zeta)$, one has
\begin{align*}
n^k(\zeta)-n^{k-1}(\zeta)=&\int_{W(\zeta)}^\infty [F_+^k(\zeta,\varepsilon)-F_+^{k-1}(\zeta,\varepsilon)]
\frac{\text{d}\varepsilon}{\sqrt{2[\varepsilon-W(\zeta)]}}
\\
&+\int_{W(\zeta)}^\infty [F_-^k(\zeta,\varepsilon)-F_-^{k-1}(\zeta,\varepsilon)]
\frac{\text{d}\varepsilon}{\sqrt{2[\varepsilon-W(\zeta)]}},
\end{align*}
so that
\[
F^k_\pm\ge F^{k-1}_\pm \implies n^k \ge n^{k-1} \implies F_\pm^{k+1}\ge F_\pm^k.
\]
Besides
\begin{align*}
& F_+^1(\zeta,\varepsilon)=\bm{1}_{\varepsilon>0}\, \theta(\zeta_a(\varepsilon),\zeta;\varepsilon)\,
\theta(\zeta_a(\varepsilon),\infty;\varepsilon) F_\infty (-\sqrt{2\varepsilon} ) \ge 0=F_+^0(\zeta,\varepsilon),
\\
& F_-^1(\zeta,\varepsilon)=\bm{1}_{\varepsilon>0}\, \theta(\zeta,\infty;\varepsilon)
F_\infty (-\sqrt{2\varepsilon} ) \ge 0=F_-^0(\zeta,\varepsilon),
\end{align*}
so that, by induction, we conclude that
\begin{align*}
& 0=F^0_\pm(\zeta,\varepsilon) \le F^1_\pm(\zeta,\varepsilon) \le F^2_\pm(\zeta,\varepsilon)
\le \ldots \le F^k_\pm(\zeta,\varepsilon) \le \ldots,
\\
& 0=n^0(\zeta) \le n^1(\zeta) \le n^2(\zeta) \le \ldots \le n^k(\zeta) \le \ldots.
\end{align*}
It remains to prove the upper bound. Here again, we proceed by induction. Clearly
\[
F^0_\pm=0 \implies n^0=0 \implies F^0_\pm \le Ae^{-\varepsilon}/\sqrt{2\pi}
\;\;\; \text{and}\;\;\; n^0\mathcal{M}(\zeta,\varepsilon) \le Ae^{-\varepsilon}/\sqrt{2\pi}.
\]
Next we prove that
\[
n^{k-1}(\zeta)\mathcal{M}(\zeta,\varepsilon) \le Ae^{-\varepsilon}/\sqrt{2\pi}
\implies F^k_\pm \le Ae^{-\varepsilon}/\sqrt{2\pi}.
\]
Indeed, by the use of \eqref{est-intK+} and \eqref{est-intK-}, the following
inequalities follow from \eqref{Fpm-a} and \eqref{Fpm-b}:
\begin{align*}
F_+^k(\zeta,\varepsilon)& = \bm{1}_{\varepsilon>0}\, \bm{1}_{\zeta > \zeta_a(\varepsilon)}\,
\theta(\zeta_a(\varepsilon),\zeta;\varepsilon)\,
\theta(\zeta_a(\varepsilon),\infty;\varepsilon)F_\infty (-\sqrt{2\varepsilon} )
\\
& \quad +\int_0^\infty K_+(\zeta,s,\varepsilon)n^{k-1}(s)\mathcal{M}(s,\varepsilon)\text{d}s
\\
& \le \bm{1}_{\varepsilon>0}\, \bm{1}_{\zeta > \zeta_a(\varepsilon)}\,
\theta(\zeta_a(\varepsilon),\zeta;\varepsilon)\,
\theta(\zeta_a(\varepsilon),\infty;\varepsilon)Ae^{-\varepsilon}/\sqrt{2\pi}
\\
& \quad +\left[ \int_0^\infty K_+(\zeta,s,\varepsilon)\text{d}s\right] Ae^{-\varepsilon}/\sqrt{2\pi}
\\
& \le \bm{1}_{\varepsilon>0}\, \bm{1}_{\zeta > \zeta_a(\varepsilon)}\,
\theta(\zeta_a(\varepsilon),\zeta;\varepsilon)\,
\theta(\zeta_a(\varepsilon),\infty;\varepsilon)Ae^{-\varepsilon}/\sqrt{2\pi}
\\
& \quad + [\bm{1}_{W(\zeta)<\varepsilon}-\bm{1}_{\varepsilon>0}\bm{1}_{\zeta>\zeta_a(\varepsilon)}
\theta(\zeta_a(\varepsilon),\zeta;\varepsilon)\, \theta(\zeta_a(\varepsilon),\infty;\varepsilon) ] Ae^{-\varepsilon}/\sqrt{2\pi}
\\
& = \bm{1}_{W(\zeta)<\varepsilon} A e^{-\varepsilon}/\sqrt{2\pi}
\\
& \le Ae^{-\varepsilon}/\sqrt{2\pi},
\end{align*}
while
\begin{align*}
F_-^k(\zeta,\varepsilon)& = \bm{1}_{\varepsilon>0}\, \bm{1}_{\zeta > \zeta_a(\varepsilon)}\,
\theta(\zeta,\infty;\varepsilon) F_\infty (-\sqrt{2\varepsilon})
+\int_0^\infty K_-(\zeta,s,\varepsilon)n^{k-1}(s)\mathcal{M}(s,\varepsilon)\text{d}s
\\
& \le \bm{1}_{\varepsilon>0}\, \bm{1}_{\zeta > \zeta_a(\varepsilon)}\,
\theta(\zeta,\infty;\varepsilon)Ae^{-\varepsilon}/\sqrt{2\pi}
+\left[ \int_0^\infty K_-(\zeta,s,\varepsilon)\text{d}s\right] Ae^{-\varepsilon}/\sqrt{2\pi}
\\
& \le \bm{1}_{\varepsilon>0}\, \bm{1}_{\zeta > \zeta_a(\varepsilon)}\,
\theta(\zeta,\infty;\varepsilon)Ae^{-\varepsilon}/\sqrt{2\pi}
\\
& \quad
+ [ \bm{1}_{W(\zeta)<\varepsilon}-\bm{1}_{\varepsilon>0}\bm{1}_{\zeta>\zeta_a(\varepsilon)}\theta(\zeta,\infty,\varepsilon) ]
Ae^{-\varepsilon}/\sqrt{2\pi}
\\
& = \bm{1}_{W(\zeta)<\varepsilon} A e^{-\varepsilon}/\sqrt{2\pi}
\\
& \le Ae^{-\varepsilon}/\sqrt{2\pi},
\end{align*}
where $\bm{1}_{\varepsilon > 0}$ in \eqref{Fpm} has been replaced
with the more explicit representation
$\bm{1}_{\varepsilon>0}\,\bm{1}_{\zeta > \zeta_a(\varepsilon)}$ in consistency
with the expressions \eqref{est-intK+} and \eqref{est-intK-}.

It remains to prove that
\[
F^k_\pm \le Ae^{-\varepsilon}/\sqrt{2\pi} \implies
n^k(\zeta) \mathcal{M}(\zeta,\varepsilon) \le Ae^{-\varepsilon}/\sqrt{2\pi}.
\]
Observe that
\begin{align*}
n^k(\zeta)& = \int_{W(\zeta)}^\infty [F^k_-(\zeta,\varepsilon)+F^k_+(\zeta,\varepsilon)] \frac{\text{d}\varepsilon}
{\sqrt{2[\varepsilon-W(\zeta)]}}
\\
& \le \frac{2A}{\sqrt{2\pi}}\int_{W(\zeta)}^\infty\frac{e^{-\varepsilon}\text{d}\varepsilon}{\sqrt{2[\varepsilon-W(\zeta)]}}
= \frac{2Ae^{-W(\zeta)}}{\sqrt{2\pi}}\int_0^\infty\frac{e^{-u}\text{d}u}{\sqrt{2u}}
\\
& = \frac{Ae^{-W(\zeta)}}{\sqrt{\pi}} \Gamma \left(\frac{1}{2} \right)
= Ae^{-W(\zeta)},
\end{align*}
where $\Gamma (x) = \int_0^\infty s^{x-1} e^{-s} \text{d}s$ is Euler’s gamma function.
Thus, it follows that
\[
F^k_\pm\le Ae^{-\varepsilon}/\sqrt{2\pi} \implies
n^k(\zeta) \le Ae^{-W(\zeta)}\implies n^k(\zeta)\mathcal{M}(\zeta,\varepsilon) \le Ae^{-\varepsilon}/\sqrt{2\pi}.
\]
This completes the proof of monotonicity of the approximating sequence.
\end{proof}

\subsubsection{Existence proof for problem \eqref{HS-3}}

Using Lemma \ref{lem-monoton} and applying the monotone convergence
theorem, we find that $F^k_\pm(\zeta,\varepsilon) \to F_\pm(\zeta,\varepsilon)$ 
and $n^k(\zeta) \to n(\zeta)$ as $k \to\infty$. Then, we pass to the limit
as $k \to\infty$ in the formulas \eqref{HS-3-it-intg} to obtain
\begin{align*}
F_+(\zeta,\varepsilon)& =\theta(\zeta_a(\varepsilon),\zeta;\varepsilon)
\Bigg(F_-(\zeta_a(\varepsilon),\varepsilon)
\\
& \qquad \qquad \;\;\;
+\int_{\zeta_a(\varepsilon)}^\zeta \theta(s,\zeta_a(\varepsilon);\varepsilon)
\frac{n(s)\mathcal{M}(s,\varepsilon)\text{d}s}{\tau(s)\sqrt{2[\varepsilon-W(s)]}}\Bigg),
\\
F_-(\zeta,\varepsilon)& =\bm{1}_{\varepsilon>0}\, \theta(\zeta,\infty;\varepsilon)
\Bigg(F_\infty (-\sqrt{2\varepsilon} )
\\
& \qquad \qquad \;\;\;
+\int_\zeta^\infty\theta(\infty,s;\varepsilon)\frac{n(s)\mathcal{M}(s,\varepsilon)\text{d}s}
{\tau(s)\sqrt{2[\varepsilon-W(s)]}}\Bigg)
\\
& \quad +\bm{1}_{W_\text{min}<\varepsilon<0}\, \theta(\zeta,\zeta_b(\varepsilon);\varepsilon)
\Bigg(F_+(\zeta_b(\varepsilon),\varepsilon)
\\
& \qquad \qquad \;\;\;
+\int_\zeta^{\zeta_b(\varepsilon)}\theta(\zeta_b(\varepsilon),s;\varepsilon)
\frac{n(s)\mathcal{M}(s,\varepsilon)\text{d}s}
{\tau(s)\sqrt{2[\varepsilon-W(s)]}}\Bigg).
\end{align*}
It is readily seen that these equations are equivalent to \eqref{HS-3} and
\eqref{HS-3-continuity}. In other words, the nondecreasing sequence
$\{ F^k_\pm \}$ converges to a solution of the problem \eqref{HS-3}, which
shows the existence of a solution $F_\pm$ 
of \eqref{HS-3} satisfying the inequality 
\[
0\le F_\pm(\zeta,\varepsilon) \le Ae^{-\varepsilon}/\sqrt{2\pi}.
\]
As for the uniqueness of such a solution, it has already been proved in
Lemma \ref{lem-unique}.

\subsubsection{Convergence as $\zeta\to\infty$}

We investigate the limit of $F_\pm (\zeta, \varepsilon)$ as $\zeta \to \infty$
on the basis of the expressions \eqref{Fpm} and \eqref{Kpm}. Therefore,
it is sufficient to consider the case of $\varepsilon > 0$.
Let us recall the definition \eqref{theta} of the function $\theta(a,b;\varepsilon)$
and its properties \eqref{theta-prop-1} and \eqref{diff-theta}, as well as
the definitions \eqref{HS-3-c} for $\mathcal{M}(s,\varepsilon)$ and \eqref{def-mu}
for $\mu (s,\varepsilon)$.

Choose $\varepsilon > 0$ to be kept fixed. Then, it follows from \eqref{Kpm-a} that
\begin{align*}
0 & \le K_+(\zeta,s,\varepsilon)
\\
& = \bm{1}_{\zeta_a(\varepsilon)<s<\zeta}\theta(s,\zeta;\varepsilon)\mu(s,\varepsilon)
\\
& \quad +\bm{1}_{\zeta>\zeta_a(\varepsilon)}\bm{1}_{s>\zeta_a(\varepsilon)}
\theta(\zeta_a(\varepsilon),\zeta;\varepsilon)\theta(\zeta_a(\varepsilon),s;\varepsilon)\mu(s,\varepsilon)
\\
& \le \bm{1}_{\zeta>\zeta_a(\varepsilon)}\bm{1}_{s>\zeta_a(\varepsilon)}\, 2\mu(s,\varepsilon),
\end{align*}
and
\begin{align*}
& \lim_{\zeta\to+\infty}K_+(\zeta,s,\varepsilon)
\\
& \quad=\bm{1}_{s>\zeta_a(\varepsilon)}
\mu(s,\varepsilon) [\theta(s,+\infty;\varepsilon)+\theta(\zeta_a(\varepsilon),+\infty;\varepsilon)\,
\theta(\zeta_a(\varepsilon),s;\varepsilon)].
\end{align*}
Similarly, \eqref{Kpm-b} leads to 
\begin{align*}
0\le  K_-(\zeta,s,\varepsilon)=\bm{1}_{s>\zeta>\zeta_a(\varepsilon)}\theta(\zeta,s;\varepsilon)\mu(s,\varepsilon)
\le \bm{1}_{s>\zeta>\zeta_a(\varepsilon)}\mu(s,\varepsilon),
\end{align*}
so that
$$
\lim_{\zeta\to+\infty}K_-(\zeta,s,\varepsilon)=0.
$$

Assume that 
$$
0\le F_\infty(-\sqrt{2\varepsilon})\le Ae^{-\varepsilon}/\sqrt{2\pi},
$$
so that we know from Lemma \ref{lem-monoton} that
$$
0\le F_\pm(\zeta,\varepsilon) \le Ae^{-\varepsilon}/\sqrt{2\pi}\quad
\text{ and }\quad 0\le n(\zeta)\le Ae^{-W(\zeta)}.
$$
Then
\begin{align*}
0 & \le K_+(\zeta,s,\varepsilon)n(s)\mathcal{M}(s,\varepsilon)
\\
& \le 2\bm{1}_{\zeta>\zeta_a(\varepsilon)}\bm{1}_{s>\zeta_a(\varepsilon)}\mu(s,\varepsilon) Ae^{-W(s)}\tfrac1{\sqrt{2\pi}}e^{-\varepsilon+W(s)}
\\
& = A\sqrt{\tfrac2\pi}e^{-\varepsilon}\bm{1}_{\zeta>\zeta_a(\varepsilon)}\bm{1}_{s>\zeta_a(\varepsilon)}\mu(s,\varepsilon)
\\
& \le A\sqrt{\tfrac2\pi}\bm{1}_{\zeta>\zeta_a(\varepsilon)}\bm{1}_{s>\zeta_a(\varepsilon)}\mu(s,\varepsilon),
\end{align*}
while
\begin{align*}
0 &\le K_-(\zeta,s,\varepsilon)n(s)\mathcal{M}(s,\varepsilon)
\\
& \le \bm{1}_{s>\zeta>\zeta_a(\varepsilon)}\mu(s,\varepsilon)Ae^{-W(s)}\tfrac1{\sqrt{2\pi}}e^{-\varepsilon+W(s)}
\\
& =\tfrac{A}{\sqrt{2\pi}}e^{-\varepsilon}\bm{1}_{s>\zeta>\zeta_a(\varepsilon)}\mu(s,\varepsilon)
\\
& \le \tfrac{A}{\sqrt{2\pi}}\bm{1}_{s>\zeta>\zeta_a(\varepsilon)}\mu(s,\varepsilon).
\end{align*}
Since
$$
\int_0^\infty \bm{1}_{s>\zeta_a(\varepsilon)}\mu(s,\varepsilon)\text{d}s<\infty,
$$
we conclude by dominated convergence that, for each $\varepsilon>0$, one has
\begin{align*}
& \lim_{\zeta\to+\infty}\int_0^\infty K_+(\zeta,s,\varepsilon)n(s)\mathcal{M}(s,\varepsilon)\text{d}s
\\
& \qquad =\int_0^\infty \bm{1}_{s>\zeta_a(\varepsilon)}\mu(s,\varepsilon)
[\theta(s,+\infty,\varepsilon)
\\
& \qquad \qquad \qquad \qquad \qquad
+\theta(\zeta_a(\varepsilon),+\infty,\varepsilon)
\theta(\zeta_a(\varepsilon),s,\varepsilon)]n(s)\mathcal{M}(s,\varepsilon)\text{d}s,
\\
& \lim_{\zeta\to+\infty}\int_0^\infty K_-(\zeta,s,\varepsilon)n(s)\mathcal{M}(s,\varepsilon)\text{d}s=0.
\end{align*}
Hence, for each $\varepsilon>0$, 
\begin{subequations}\label{limits}
\begin{align}
& \lim_{\zeta\to+\infty}F_+(\zeta,\varepsilon)
\nonumber \\
& \qquad = \theta(\zeta_a(\varepsilon),\infty;\varepsilon)^2F_\infty (-\sqrt{2\varepsilon})
\nonumber \\
& \qquad \quad +\int_{\zeta_a(\varepsilon)}^\infty \mu(s,\varepsilon)
[\theta(s,+\infty,\varepsilon)
\nonumber \\
& \qquad \qquad \qquad \qquad \qquad
+\theta(\zeta_a(\varepsilon),+\infty,\varepsilon)\theta(\zeta_a(\varepsilon),s,\varepsilon)]n(s)\mathcal{M}(s,\varepsilon)\text{d}s,
\label{limits-a} \\
& \lim_{\zeta\to+\infty}F_-(\zeta,\varepsilon)=F_\infty (-\sqrt{2\varepsilon}).
\label{limits-b}
\end{align}
\end{subequations}
Equation \eqref{limits-b} confirms that the boundary condition \eqref{HS-2-d}
in the problem \eqref{HS-2} is satisfied, while \eqref{limits-a}
shows that the limit of the solution $F(\zeta, c_z)$ to the problem
\eqref{HS-2} as $\zeta \to \infty$ exists for $c_z > 0$.

\subsection{Proof of exponential convergence of approximating sequence $\{F_\pm^k\}$}\label{subsec:expornential}

We finally prove Theorem \ref{thrm-exp-conv}. 
Let us return to the integral transformation \eqref{Fpm}. Multiplying each side of both equalities
by $\mu(\zeta,\varepsilon)\tau(\zeta)$ and integrating in $\varepsilon>W(\zeta)$, one finds that
\begin{align}\label{IntEq-nk}
n^k(\zeta)=\int_0^\infty\mathcal{K}(\zeta,s)n^{k-1}(s)\text{d}s+N(\zeta)\,,\qquad n^0=0,
\end{align}
where
\begin{align}\label{Def-cK}
\mathcal{K}(\zeta,s):=\tau(\zeta)\int_{W(\zeta)}^\infty
[K_+(\zeta,s,\varepsilon)+K_-(\zeta,s,\varepsilon)]\mu(\zeta,\varepsilon)
\mathcal{M}(s,\varepsilon)\text{d}\varepsilon,
\end{align}
and
\begin{align}\label{Def-N}
& N(\zeta):=\tau(\zeta)\int_{W(\zeta)}^\infty
\bm{1}_{\varepsilon>0}\,
[\theta(\zeta_a(\varepsilon),\zeta;\varepsilon)\theta(\zeta_a(\varepsilon),\infty;\varepsilon)
\nonumber \\
& \qquad \qquad \qquad \qquad \qquad \quad
+\theta(\zeta,\infty;\varepsilon)]
F_\infty (-\sqrt{2\varepsilon}) \mu(\zeta,\varepsilon)\text{d}\varepsilon.
\end{align}

The following two things suggest considering the quantity
$$
\int_0^\infty n^k(\zeta)\frac{\text{d}\zeta}{\tau(\zeta)}:
$$
\begin{itemize}
\item
if $n^k(\zeta)$ converges to a limit $n^k_\infty>0$ as $\zeta\to+\infty$,
the weight $1/\tau(\zeta)$, known to be integrable on the half-line, will make
the function $\zeta\mapsto n^k(\zeta)/\tau(\zeta)$ also integrable; and
\item
there is the prefactor $\tau(\zeta)$ in the definition \eqref{Def-cK}
of $\mathcal{K}(\zeta,s)$ as well as in the definition \eqref{Def-N} of $N(\zeta)$.
\end{itemize}

Thus, multiplying both sides of \eqref{IntEq-nk} by $1/\tau(\zeta)$, integrating in
$\zeta$ over the half-line, and exchanging the order of integration in $\zeta$
and in $s$ in the first integral on the right-hand side of \eqref{IntEq-nk},
we arrive at the equality
\begin{align}\label{n^k-int}
\int_0^\infty n^k(\zeta)\frac{\text{d}\zeta}{\tau(\zeta)}
=\int_0^\infty\left(\int_0^\infty\mathcal{K}(\zeta,s)\frac{\tau(s)}{\tau(\zeta)}\text{d}
\zeta\right)n^{k-1}(s)\frac{\text{d}s}{\tau(s)}+\int_0^\infty N(\zeta)\frac{\text{d}\zeta}{\tau(\zeta)}.
\end{align}
Therefore, we are left with the task of computing
\begin{align}\label{Int-cK}
& \int_0^\infty\! \mathcal{K}(\zeta,s) \frac{\text{d}\zeta}{\tau(\zeta)}
\nonumber \\
& \qquad = \int_0^\infty\!\!\int_{W(\zeta)}^\infty
[K_+(\zeta,s,\varepsilon)+K_-(\zeta,s,\varepsilon)]
\mu(\zeta,\varepsilon)\mathcal{M}(s,\varepsilon)\text{d}\varepsilon \text{d}\zeta
\nonumber \\
& \qquad = \int_0^\infty\!\mathcal{M}(s,\varepsilon) \left(\int_{\zeta_a(\varepsilon)}^\infty
[K_+(\zeta,s,\varepsilon)\!+\!K_-(\zeta,s,\varepsilon)]
\mu(\zeta,\varepsilon)\text{d}\zeta\right) \text{d}\varepsilon 
\nonumber \\
& \qquad \quad +\int_{W_\text{min}}^0\!\mathcal{M}(s,\varepsilon)
\left(\int_{\zeta_a(\varepsilon)}^{\zeta_b(\varepsilon)}
[K_+(\zeta,s,\varepsilon)\!+\!K_-(\zeta,s,\varepsilon)]\mu(\zeta,\varepsilon)
\text{d}\zeta\right)\text{d}\varepsilon.
\end{align}
\subsubsection{Computing the inner integrals in \eqref{Int-cK}}

We recall from \eqref{Kpm-a} and \eqref{Kpm-b} that
\begin{align}\label{muK+}
& \mu(\zeta,\varepsilon)K_+(\zeta,s,\varepsilon)
\nonumber \\
& \quad
=\bm{1}_{\varepsilon>0}\bm{1}_{\zeta>\zeta_a(\varepsilon)}\bm{1}_{s>\zeta_a(\varepsilon)}
[-\partial_\zeta\theta(\zeta_a(\varepsilon),\zeta;\varepsilon)]
\theta(\zeta_a(\varepsilon),s;\varepsilon)\mu(s,\varepsilon)
\nonumber \\
& \quad \quad
+\bm{1}_{W_\text{min}<\varepsilon<0}\bm{1}_{\zeta_a(\varepsilon)<\zeta<\zeta_b(\varepsilon)}
\bm{1}_{\zeta_a(\varepsilon)<s<\zeta_b(\varepsilon)}
[-\partial_\zeta\theta(\zeta_a(\varepsilon),\zeta;\varepsilon)]
\nonumber \\
& \qquad \qquad \qquad \qquad \qquad \qquad \qquad \times
\tfrac{[\theta(s,\zeta_b(\varepsilon);\varepsilon)+\theta(\zeta_b(\varepsilon),s;\varepsilon)]
\mu(s,\varepsilon)}{\theta(\zeta_b(\varepsilon),\zeta_a(\varepsilon);\varepsilon)-\theta(\zeta_a(\varepsilon),\zeta_b(\varepsilon);\varepsilon)}
\nonumber \\
& \quad \quad
+(\bm{1}_{\varepsilon>0}\bm{1}_{\zeta_a(\varepsilon)<s<\zeta}+\bm{1}_{W_\text{min}<\varepsilon<0}
\bm{1}_{\zeta_a(\varepsilon)<s<\zeta<\zeta_b(\varepsilon)})
[-\partial_\zeta\theta(s,\zeta;\varepsilon)]\mu(s,\varepsilon),
\end{align}
and
\begin{align}\label{muK-}
& \mu(\zeta,\varepsilon)K_-(\zeta,s,\varepsilon)
\nonumber \\
& \quad
=(\bm{1}_{\varepsilon>0}\bm{1}_{\zeta_a(\varepsilon)<\zeta<s}+\bm{1}_{W_\text{min}<\varepsilon<0}
\bm{1}_{\zeta_a(\varepsilon)<\zeta<s<\zeta_b(\varepsilon)})\, \partial_\zeta\theta(\zeta,s;\varepsilon)\mu(s,\varepsilon)
\nonumber \\
& \quad \quad
+\bm{1}_{W_\text{min}<\varepsilon<0}\bm{1}_{\zeta_a(\varepsilon)<\zeta<\zeta_b(\varepsilon)}
\bm{1}_{\zeta_a(\varepsilon)<s<\zeta_b(\varepsilon)}\, \partial_\zeta\theta(\zeta,\zeta_b(\varepsilon);\varepsilon)
\nonumber \\
& \qquad \qquad \qquad \qquad \qquad \qquad \qquad \times
\tfrac{[\theta(s,\zeta_a(\varepsilon);\varepsilon)+\theta(\zeta_a(\varepsilon),s;\varepsilon)]
\mu(s,\varepsilon)}{\theta(\zeta_b(\varepsilon),\zeta_a(\varepsilon);\varepsilon)-\theta(\zeta_a(\varepsilon),\zeta_b(\varepsilon);\varepsilon)}.
\end{align}
Observe that we have used \eqref{diff-theta} to transform the terms
$$
\tfrac{\partial_s\theta(s,\zeta_b(\varepsilon);\varepsilon)-\partial_s\theta(\zeta_b(\varepsilon),s;\varepsilon)}
{\theta(\zeta_b(\varepsilon),\zeta_a(\varepsilon);\varepsilon)-\theta(\zeta_a(\varepsilon),\zeta_b(\varepsilon);\varepsilon)}
\;\; \text{ and } \;\;
\tfrac{\partial_s\theta(s,\zeta_a(\varepsilon);\varepsilon)-\partial_s\theta(\zeta_a(\varepsilon),s;\varepsilon)}
{\theta(\zeta_b(\varepsilon),\zeta_a(\varepsilon);\varepsilon)-\theta(\zeta_a(\varepsilon),\zeta_b(\varepsilon);\varepsilon)}
$$
into
$$
\tfrac{[\theta(s,\zeta_b(\varepsilon);\varepsilon)+\theta(\zeta_b(\varepsilon),s;\varepsilon)]\mu(s,\varepsilon)}
{\theta(\zeta_b(\varepsilon),\zeta_a(\varepsilon);\varepsilon)-\theta(\zeta_a(\varepsilon),\zeta_b(\varepsilon);\varepsilon)}
\;\; \text{ and } \;\;
\tfrac{[\theta(s,\zeta_a(\varepsilon);\varepsilon)+\theta(\zeta_a(\varepsilon),s;\varepsilon)]\mu(s,\varepsilon)}
{\theta(\zeta_b(\varepsilon),\zeta_a(\varepsilon);\varepsilon)-\theta(\zeta_a(\varepsilon),\zeta_b(\varepsilon);\varepsilon)},
$$
respectively, and the terms
$$
\mu(\zeta,\varepsilon)\theta(\zeta_a(\varepsilon),\zeta;\varepsilon)
\;\; \text{ and } \;\;
\mu(\zeta;\varepsilon)\theta(\zeta,\zeta_b(\varepsilon);\varepsilon)
$$
into
$$
-\partial_\zeta\theta(\zeta_a(\varepsilon),\zeta;\varepsilon)
\;\; \text{ and } \;\;
\partial_\zeta\theta(\zeta,\zeta_b(\varepsilon);\varepsilon),
$$
respectively.

Then, for $\varepsilon>0$, one has
\begin{subequations}\label{K-int-e>0}
\begin{align}
& \int_{\zeta_a(\varepsilon)}^\infty K_+(\zeta,s,\varepsilon)\mu(\zeta,\varepsilon)\text{d}\zeta
\nonumber \\
& \qquad \;\;\;
=\bm{1}_{\varepsilon>0}\bm{1}_{s>\zeta_a(\varepsilon)}
[1-\theta(\zeta_a(\varepsilon),\infty;\varepsilon)] \theta(\zeta_a(\varepsilon),s;\varepsilon)\mu(s,\varepsilon)
\nonumber \\
& \qquad \;\;\; \quad
+\bm{1}_{\varepsilon>0}\bm{1}_{\zeta_a(\varepsilon)<s} [1-\theta(s,\infty;\varepsilon)] \mu(s,\varepsilon),
\\
& \int_{\zeta_a(\varepsilon)}^\infty K_-(\zeta,s,\varepsilon)\mu(\zeta,\varepsilon)\text{d}\zeta
=\bm{1}_{\varepsilon>0}\bm{1}_{\zeta_a(\varepsilon)<s} [1-\theta(\zeta_a(\varepsilon),s;\varepsilon)] \mu(s,\varepsilon),
\end{align}
\end{subequations}
while, for $W_\text{min}<\varepsilon<0$, 

\begin{subequations}\label{K-int-e<0}
\begin{align}
& \int_{\zeta_a(\varepsilon)}^{\zeta_b(\varepsilon)} K_+(\zeta,s,\varepsilon)\mu(\zeta,\varepsilon)\text{d}\zeta
\nonumber \\
& \qquad \;\;\;
=\bm{1}_{W_\text{min}<\varepsilon<0}\bm{1}_{\zeta_a(\varepsilon)<s<\zeta_b(\varepsilon)}
[1-\theta(s,\zeta_b(\varepsilon);\varepsilon)]\mu(s,\varepsilon)
\nonumber \\
& \qquad \;\;\; \quad
+\bm{1}_{W_\text{min}<\varepsilon<0}\bm{1}_{\zeta_a(\varepsilon)<s<\zeta_b(\varepsilon)}
[1-\theta(\zeta_a(\varepsilon),\zeta_b(\varepsilon);\varepsilon)]
\nonumber \\
& \qquad \qquad \qquad \qquad \qquad \qquad \qquad \times
\tfrac{[\theta(s,\zeta_b(\varepsilon);\varepsilon)+\theta(\zeta_b(\varepsilon),s;\varepsilon)]\mu(s,\varepsilon)}
{\theta(\zeta_b(\varepsilon),\zeta_a(\varepsilon);\varepsilon)-\theta(\zeta_a(\varepsilon),\zeta_b(\varepsilon);\varepsilon)},
\\
& \int_{\zeta_a(\varepsilon)}^{\zeta_b(\varepsilon)} K_-(\zeta,s,\varepsilon)\mu(\zeta,\varepsilon)\text{d}\zeta
\nonumber \\
& \qquad \;\;\;
=\bm{1}_{W_\text{min}<\varepsilon<0}\bm{1}_{\zeta_a(\varepsilon)<s<\zeta_b(\varepsilon)}
[1-\theta(\zeta_a(\varepsilon),s;\varepsilon)]\mu(s,\varepsilon)
\nonumber \\
& \qquad \;\;\; \quad
+\bm{1}_{W_\text{min}<\varepsilon<0}\bm{1}_{\zeta_a(\varepsilon)<s<\zeta_b(\varepsilon)}
[1-\theta(\zeta_a(\varepsilon),\zeta_b(\varepsilon);\varepsilon)]
\nonumber \\
& \qquad \qquad \qquad \qquad \qquad \qquad \qquad \times
\tfrac{[\theta(s,\zeta_a(\varepsilon);\varepsilon)+\theta(\zeta_a(\varepsilon),s;\varepsilon]
\mu(s,\varepsilon)}{\theta(\zeta_b(\varepsilon),\zeta_a(\varepsilon);\varepsilon)-\theta(\zeta_a(\varepsilon),\zeta_b(\varepsilon);\varepsilon)}.
\end{align}
\end{subequations}
We here note that the integral in $\zeta$ over $(\zeta_a(\varepsilon), \infty)$ and
that over $(\zeta_a(\varepsilon), \zeta_b(\varepsilon))$ in \eqref{Int-cK} are
expressed in the unified form
\begin{align}\label{unified}
\int_0^\infty [K_+(\zeta,s,\varepsilon)\!+\!K_-(\zeta,s,\varepsilon)]
\mu(\zeta,\varepsilon)\text{d}\zeta,
\end{align}
because of the expressions \eqref{muK+} and \eqref{muK-}. Using \eqref{K-int-e>0} and \eqref{K-int-e<0},
this integral is expressed as
\begin{align}\label{IntK++K-}
& \int_0^\infty [K_+(\zeta,s,\varepsilon)+K_-(\zeta,s,\varepsilon)]\mu(\zeta,\varepsilon)\text{d}\zeta
\nonumber \\
& \quad
=\bm{1}_{\varepsilon>0}\bm{1}_{\zeta_a(\varepsilon)<s}[1-\theta(\zeta_a(\varepsilon),\infty;\varepsilon)
\theta(\zeta_a(\varepsilon),s;\varepsilon)+1-\theta(s,\infty;\varepsilon)]\mu(s,\varepsilon)
\nonumber \\
& \qquad
+\bm{1}_{W_\text{min}<\varepsilon<0}\bm{1}_{\zeta_a(\varepsilon)<s<\zeta_b(\varepsilon)}
[1-\theta(s,\zeta_b(\varepsilon);\varepsilon)+1-\theta(\zeta_a(\varepsilon),s;\varepsilon)]\mu(s,\varepsilon)
\nonumber \\
& \qquad
+\bm{1}_{W_\text{min}<\varepsilon<0}\bm{1}_{\zeta_a(\varepsilon)<s<\zeta_b(\varepsilon)}
[1-\theta(\zeta_a(\varepsilon),\zeta_b(\varepsilon);\varepsilon)]
\nonumber \\
& \qquad \qquad \qquad
\times\tfrac{[\theta(s,\zeta_b(\varepsilon);\varepsilon)+\theta(s,\zeta_a(\varepsilon);\varepsilon)
+\theta(\zeta_b(\varepsilon),s;\varepsilon) +\theta(\zeta_a(\varepsilon),s;\varepsilon)]\mu(s,\varepsilon)}
{\theta(\zeta_b(\varepsilon),\zeta_a(\varepsilon);\varepsilon)-\theta(\zeta_a(\varepsilon),\zeta_b(\varepsilon);\varepsilon)}.
\end{align}

We shall simplify the last two lines significantly: set 
\[
X:=\theta(s,\zeta_a(\varepsilon);\varepsilon)>1\quad\text{ and }\quad
\Lambda:=\theta(\zeta_b(\varepsilon),\zeta_a(\varepsilon);\varepsilon)>1
\]
for $\zeta_a(\varepsilon)<s<\zeta_b(\varepsilon)$, so that
\[
\theta(\zeta_a(\varepsilon),s;\varepsilon)=1/X,\quad
\theta(\zeta_b(\varepsilon),s;\varepsilon)=\Lambda/X,\quad
\theta(s,\zeta_b(\varepsilon);\varepsilon)=X/\Lambda.
\]
Then, one has
\begin{align*}
&
[1-\theta(s,\zeta_b(\varepsilon);\varepsilon)+1-\theta(\zeta_a(\varepsilon),s;\varepsilon)]
\\
& +[1-\theta(\zeta_a(\varepsilon),\zeta_b(\varepsilon);\varepsilon)]
\tfrac{\theta(s,\zeta_b(\varepsilon);\varepsilon)+\theta(s,\zeta_a(\varepsilon);\varepsilon)
+\theta(\zeta_b(\varepsilon),s;\varepsilon)+\theta(\zeta_a(\varepsilon),s;\varepsilon)}
{\theta(\zeta_b(\varepsilon),\zeta_a(\varepsilon);\varepsilon)-\theta(\zeta_a(\varepsilon),\zeta_b(\varepsilon);\varepsilon)}
\\
& \quad =(2-\tfrac{X}{\Lambda}-\tfrac1X)+(1-\tfrac1\Lambda)
\frac{\frac{X}\Lambda+X+\frac\Lambda{X}+\frac1X}{\Lambda-\frac1\Lambda}
\\
& \quad =(2-\tfrac{X}{\Lambda}-\tfrac1X)+(1-\tfrac1\Lambda)
\frac{(\frac{X}\Lambda+\frac1X)(1+\Lambda)}{\frac{(\Lambda+1)(\Lambda-1)}\Lambda}
\\
& \quad =(2-\tfrac{X}{\Lambda}-\tfrac1X)+(1-\tfrac1\Lambda)\frac{\frac{X}\Lambda+\frac1X}{\frac{\Lambda-1}\Lambda}
\\
& \quad =2.
\end{align*}
Therefore,
\begin{align}\label{IntK++K-Bis}
& \int_0^\infty[K_+(\zeta,s,\varepsilon)+K_-(\zeta,s,\varepsilon)]\mu(\zeta,\varepsilon)\text{d}\zeta
\nonumber \\
&\qquad  =\bm{1}_{\varepsilon>0}\bm{1}_{\zeta_a(\varepsilon)<s}
[2-\theta(\zeta_a(\varepsilon),\infty,\varepsilon)\theta(\zeta_a(\varepsilon),s;\varepsilon)
-\theta(s,\infty;\varepsilon)]\mu(s,\varepsilon)
\nonumber \\
& \qquad \quad +2\bm{1}_{W_\text{min}<\varepsilon<0}\bm{1}_{\zeta_a(\varepsilon)<s<\zeta_b(\varepsilon)}
\mu(s,\varepsilon).
\end{align}
On the other hand, setting
\[
L:=\theta(\zeta_a(\varepsilon),\infty;\varepsilon)\quad\text{ and }\quad Z:=\theta(s,\infty;\varepsilon),
\]
one has
\[
\theta(\zeta_a(\varepsilon),s;\varepsilon)=L/Z\,,\quad L<Z<1.
\]
Thus,
\[
\theta(\zeta_a(\varepsilon),\infty,\varepsilon)\theta(\zeta_a(\varepsilon),s;\varepsilon)
+\theta(s,\infty;\varepsilon)=\tfrac{L^2}Z+Z\ge 2L,
\]
since
\[
\zeta_a(\varepsilon)<s\implies 0<\theta(\zeta_a(\varepsilon),s;\varepsilon)=L/Z<1
\]
and $Z\mapsto \tfrac{L^2}Z+Z$ is increasing on $[L,1]$. Thus, it follows that
\[
2-\theta(\zeta_a(\varepsilon),\infty;\varepsilon)\theta(\zeta_a(\varepsilon),s;\varepsilon)
-\theta(s,\infty;\varepsilon)\le 2[1-\theta(\zeta_a(\varepsilon),\infty,\varepsilon)].
\]

Summarizing, we have proved the following inequality:
\begin{align}\label{IntK++K-<}
& \int_0^\infty[K_+(\zeta,s,\varepsilon)+K_-(\zeta,s,\varepsilon)]\mu(\zeta,\varepsilon)\text{d}\zeta
\nonumber \\
& \qquad \qquad
\le 2\mu(s,\varepsilon)\bm{1}_{W(s)<\varepsilon}[1-\theta(\zeta_a(\varepsilon),\infty,\varepsilon)\bm{1}_{\varepsilon>0}].
\end{align}
\subsubsection{Bounding \eqref{Int-cK}}

As noted in the process of deriving \eqref{IntK++K-},
the integrals with respect to $\zeta$ in \eqref{Int-cK} are unified in the form
\eqref{unified}, so that \eqref{Int-cK} is recast as
\begin{align}\label{Int-cK-2}
& \int_0^\infty\! \mathcal{K}(\zeta,s) \frac{\text{d}\zeta}{\tau(\zeta)}
\nonumber \\
& \qquad = \int_{W_\text{min}}^\infty\!\mathcal{M}(s,\varepsilon) \left(\int_0^\infty
[K_+(\zeta,s,\varepsilon)\!+\!K_-(\zeta,s,\varepsilon)]
\mu(\zeta,\varepsilon)\text{d}\zeta\right) \text{d}\varepsilon.
\end{align}
Next, integrating \eqref{IntK++K-<} multiplied by
$\mathcal{M}(s,\varepsilon)$ in $\varepsilon$
from $W_\text{min}$ to $\infty$ and using the resulting inequality in \eqref{Int-cK-2},
we arrive at the following inequality:
\begin{align}\label{cL<}
& \int_0^\infty \mathcal{K}(\zeta,s) \frac{\tau(s)}{\tau(\zeta)} \text{d}\zeta
\nonumber \\
& \qquad \le 2\int_{W(s)}^\infty \mathcal{M}(s,\varepsilon)\tau(s)\mu(s,\varepsilon)\text{d}\varepsilon
\nonumber \\
& \qquad \quad
-2\int_0^\infty\theta(\zeta_a(\varepsilon),\infty;\varepsilon)\mathcal{M}(s,\varepsilon)
\tau(s)\mu(s,\varepsilon)\bm{1}_{\zeta_a(\varepsilon)<s}\text{d}\varepsilon
\nonumber \\
& \qquad
=1-2\int_0^\infty\theta(\zeta_a(\varepsilon),\infty;\varepsilon)\mathcal{M}(s,\varepsilon)
\tau(s)\mu(s,\varepsilon)\bm{1}_{\zeta_a(\varepsilon)<s}\text{d}\varepsilon
\nonumber \\
& \qquad
=: \Upsilon(s).
\end{align}
Indeed,
\begin{align*}
\int_{W(s)}^\infty \mathcal{M}(s,\varepsilon)\tau(s)\mu(s,\varepsilon)\text{d}\varepsilon
& = \frac{1}{\sqrt{2\pi}}\int_{W(s)}^\infty \frac{e^{W(s)-\varepsilon}\text{d}\varepsilon}{\sqrt{2[\varepsilon-W(s)]}}
\\
& = \frac{1}{\sqrt{2\pi}}\int_0^\infty \frac{e^{-u}\text{d}u}{\sqrt{2u}}
=\frac{1}{2\sqrt{\pi}}\Gamma(\tfrac{1}{2})
=\frac12.
\end{align*}

Our aim is to find an upper bound for $\Upsilon(s)$, which amounts to finding
a lower bound for the integral in the fourth line of \eqref{cL<}, i.e.,
\[
\int_0^\infty\theta(\zeta_a(\varepsilon),\infty;\varepsilon)\mathcal{M}(s,\varepsilon)
\tau(s)\mu(s,\varepsilon)\bm{1}_{\zeta_a(\varepsilon)<s}\text{d}\varepsilon.
\]

In order to do so, we first seek a lower bound for 
\[
\theta(\zeta_a(\varepsilon),\infty,\varepsilon)=\exp\left(-\int_{\zeta_a(\varepsilon)}^\infty
\frac{\text{d}t}{\tau(t)\sqrt{2[\varepsilon-W(t)]}}\right),
\]
or, equivalently, an upper bound for the integral in the exponential.

This step will use some of the specifics of the potential, summarized
in \eqref{potential} and items (i)--(iv) in Sec.~\ref{subsec:kin-model}.
Note that the statements in Sec.~\ref{subsec:kin-model} are for the
dimensional potential, whereas the potential $W$ considered here is
dimensionless [cf.~\eqref{var-dimless} and \eqref{notation-ch}].
Therefore, we rephrase the specific properties of $W$ that will be
used in the following:
\begin{itemize}
\item $W$ is decreasing and convex downward on $(0,\zeta_\text{min})$, and tends to $+\infty$ at $0^+$;
\item $W$ is continuous and increasing on $(\zeta_\text{min},+\infty)$, and tends to $0^-$ at $+\infty$.
\end{itemize}

Thus, for each $\varepsilon>0$, the point $\zeta_a(\varepsilon)$ is uniquely defined by
$$
W(\zeta_a(\varepsilon))=\varepsilon\,,\quad 0<\zeta_a(\varepsilon)<\zeta_a(0)<\zeta_\text{min};
$$
besides
$$
W_\text{min}=W(\zeta_\text{min})=\min_{\zeta>0}W(\zeta)<0.
$$
On the other hand, we recall, from the properties of the dimensional relaxation
time $\tau_\text{ph,s}(\zeta)$ in \eqref{tauph-limit} and below in Sec.~\ref{subsec:kin-model},
that
$$
\tau(\zeta)\ge\tau(0)>0\,,\quad\text{ and }\quad \ell:=\int_0^\infty\frac{\text{d}s}{\tau(s)}<+\infty,
$$
for its dimensionless counterpart $\tau(\zeta)$ considered here.

We first decompose the integral contained in the function
$\theta(\zeta_a(\varepsilon), \infty; \varepsilon)$ as
\begin{align*}
& \int_{\zeta_a(\varepsilon)}^\infty\frac{\text{d}t}{\tau(t)\sqrt{2[\varepsilon-W(t)]}}
\\
& \qquad
=\int_{\zeta_a(\varepsilon)}^{\zeta_\text{min}}\frac{\text{d}t}{\tau(t)\sqrt{2[\varepsilon-W(t)]}}
+\int_{\zeta_\text{min}}^\infty\frac{\text{d}t}{\tau(t)\sqrt{2[\varepsilon-W(t)]}}
\\
& \qquad
=I+J,
\end{align*}
and estimate $I$ and $J$ separately.

The easiest term to bound is $J$: indeed, for $\varepsilon>0$, one has
$$
W\vert_{(\zeta_\text{min},+\infty)}\le 0\implies J\le
\int_{\zeta_\text{min}}^\infty\frac{\text{d}t}{\tau(t)\sqrt{2\varepsilon}}\le\frac{\ell}{\sqrt{2\varepsilon}}.
$$

Bounding the term $I$ is slightly more involved. We first use the convexity of $W$
on the interval $(0,\zeta_\text{min})$, which implies that
\begin{align*}
W(t) & =W\left(\tfrac{\zeta_\text{min}-t}{\zeta_\text{min}-\zeta_a(\varepsilon)}\zeta_a(\varepsilon)
+\tfrac{t-\zeta_a(\varepsilon)}{\zeta_\text{min}-\zeta_a(\varepsilon)}\zeta_\text{min}\right)
\\
& \le \tfrac{\zeta_\text{min}-t}{\zeta_\text{min}-\zeta_a(\varepsilon)}W(\zeta_a(\varepsilon))
+\tfrac{t-\zeta_a(\varepsilon)}{\zeta_\text{min}-\zeta_a(\varepsilon)}W(\zeta_\text{min}),
\end{align*}
or, equivalently,
\begin{align*}
W(t) & \le W(\zeta_a(\varepsilon))-\tfrac{W(\zeta_a(\varepsilon))-W(\zeta_\text{min})}
{\zeta_\text{min}-\zeta_a(\varepsilon)}[t-\zeta_a(\varepsilon)]
\\
& = \varepsilon -\tfrac{\varepsilon-W_\text{min}}{\zeta_\text{min}-\zeta_a(\varepsilon)}
[t-\zeta_a(\varepsilon)].
\end{align*}
The geometric interpretation of this inequality is as follows: it means that the graph
of $W$ is below the chord joining the point of coordinates $(\zeta_a(\varepsilon),\varepsilon)$
to the point of coordinates $(\zeta_\text{min},W_\text{min})$. 
This is indeed an obvious consequence of the convexity of $W$ on the interval $(0,\zeta_\text{min})$.

Thus
$$
t\in(\zeta_a(\varepsilon),\zeta_\text{min})\implies
\varepsilon-W(t)\ge\tfrac{\varepsilon-W_\text{min}}{\zeta_\text{min}-\zeta_a(\varepsilon)}[t-\zeta_a(\varepsilon)],
$$
and hence
$$
I\le\tfrac{1}{\tau(0)}\sqrt{\tfrac{\zeta_\text{min}-\zeta_a(\varepsilon)}{2(\varepsilon-W_\text{min})}}
\int_{\zeta_a(\varepsilon)}^{\zeta_\text{min}}\tfrac{\text{d}t}{\sqrt{t-\zeta_a(\varepsilon)}}
= \tfrac2{\tau(0)}\tfrac{\zeta_\text{min}-\zeta_a(\varepsilon)}{\sqrt{2(\varepsilon-W_\text{min})}}
\le\tfrac{2\zeta_\text{min}}{\tau(0)\sqrt{2\varepsilon}}\,.
$$
Therefore
$$
\int_{\zeta_a(\varepsilon)}^\infty\frac{\text{d}t}{\tau(t)\sqrt{2[\varepsilon-W(t)]}}
\le\left(\frac{2\zeta_\text{min}}{\tau(0)}+\ell\right)\frac{1}{\sqrt{2\varepsilon}},
$$
so that
\begin{align}\label{LowBndTh}
\theta(\zeta_a(\varepsilon),\infty;\varepsilon)\ge e^{-K/\sqrt{\varepsilon}},
\quad\text{ with }K:=\frac{\sqrt{2}\zeta_\text{min}}{\tau(0)}+\frac{\ell}{\sqrt{2}}.
\end{align}
Now, we insert the lower bound \eqref{LowBndTh} in
the definition of $\Upsilon(s)$ included in \eqref{cL<}, that is,
\begin{align*}
\frac{1-\Upsilon(s)}{2}
& =\int_0^\infty\theta(\zeta_a(\varepsilon),\infty;\varepsilon)
\mathcal{M}(s,\varepsilon)\tau(s)\mu(s,\varepsilon)\bm{1}_{\zeta_a(\varepsilon)<s}\text{d}\varepsilon
\\
& \ge\int_0^\infty e^{-K/\sqrt{\varepsilon}}\mathcal{M}(s,\varepsilon)\tau(s)\mu(s,\varepsilon)
\bm{1}_{\zeta_a(\varepsilon)<s}\text{d}\varepsilon
\\
& =\frac{1}{\sqrt{2\pi}}\int_0^\infty \frac{e^{-K/\sqrt{\varepsilon}}e^{W(s)-\varepsilon}
\bm{1}_{\zeta_a(\varepsilon)<s}}{\sqrt{2[\varepsilon-W(s)]}}\text{d}\varepsilon
\\
& =\frac1{\sqrt{2\pi}}\int_{W(s)^+}^\infty \frac{e^{-K/\sqrt{\varepsilon}}e^{W(s)-\varepsilon}}
{\sqrt{2[\varepsilon-W(s)]}}\text{d}\varepsilon,
\end{align*}
with the notation $X^+=\max(X,0)$ and $X^-=\max(-X,0)$.
Change variables in the last integral, setting $u=\varepsilon-W(s)$. Since
$$
W(s)=W(s)^+-W(s)^-,
$$
one has
\begin{align*}
\frac{1-\Upsilon(s)}{2} &\ge \frac{1}{\sqrt{2\pi}}
\int_{W(s)^-}^\infty\frac{e^{-K/\sqrt{u+W(s)}}e^{-u}}{\sqrt{2u}}\text{d}u
\\
&\ge \frac1{\sqrt{2\pi}}\int_{\vert W_\text{min}\vert+1}^\infty \frac{e^{-K/\sqrt{u+W(s)}}e^{-u}}{\sqrt{2u}}\text{d}u
\end{align*}
since $W(s)^-\le \vert W_\text{min} \vert \le \vert W_\text{min} \vert +1$. 

Next, observe that the function $v\mapsto e^{-K/\sqrt{v}}$ is increasing on $(0,+\infty)$, so that
$$
u\ge\vert W_\text{min}\vert+1\implies
e^{-K/\sqrt{u+W(s)}}\ge e^{-K/\sqrt{1+\vert W_\text{min}\vert+W(s)}}\ge e^{-K}.
$$
Therefore
\begin{align*}
\frac{1-\Upsilon(s)}{2} & =\int_0^\infty\theta(\zeta_a(\varepsilon),\infty;\varepsilon)
\mathcal{M}(s,\varepsilon)\tau(s)\mu(s,\varepsilon)\bm{1}_{\zeta_a(\varepsilon)<s}\text{d}\varepsilon
\\
& \ge \frac{1}{\sqrt{2\pi}}\int_{\vert W_\text{min}\vert+1}^\infty\frac{e^{-K}e^{-u}}{\sqrt{2u}}\text{d}u>0.
\end{align*}

Summarizing, we have proved the following inequality
\begin{align*}
0 &\le\int_0^\infty\!\mathcal{K}(\zeta,s)\frac{\tau(s)}{\tau(\zeta)}\text{d}\zeta
\\
& \le 1-2\int_0^\infty\theta(\zeta_a(\varepsilon),\infty,\varepsilon)
\mathcal{M}(s,\varepsilon)\tau(s)\mu(s,\varepsilon)\bm{1}_{\zeta_a(\varepsilon)<s}\text{d}\varepsilon
\\
& = \Upsilon(s) \le \mathcal{L},
\end{align*}
where
\[
\mathcal{L} = 1-\frac{e^{-K}}{\sqrt{\pi}}\int_{\vert W_\text{min}\vert+1}^\infty\frac{e^{-u}}{\sqrt{u}}\text{d}u<1.
\]
\subsubsection{Exponential convergence of $n^k$ and of $F^k_\pm$}

Now we recall \eqref{n^k-int} and pass to the limit as $k \to \infty$
in both sides of this equality.
Since $n^k\to n$ a.e. on $(0,+\infty)$ by monotone convergence (cf.~Lemma \ref{lem-monoton}),
one has
\begin{align*}
\int_0^\infty n(\zeta)\frac{\text{d}\zeta}{\tau(\zeta)}
=\int_0^\infty\left(\int_0^\infty\mathcal{K}(\zeta,s)\frac{\tau(s)}{\tau(\zeta)}\text{d}
\zeta\right)n(s)\frac{\text{d}s}{\tau(s)}+\int_0^\infty N(\zeta)\frac{\text{d}\zeta}{\tau(\zeta)},
\end{align*}
so that
\begin{align*}
0 & \le\int_0^\infty[n(\zeta) -n^k(\zeta)]\frac{\text{d}\zeta}{\tau(\zeta)}
\\
& = \int_0^\infty \left(\int_0^\infty\mathcal{K}(\zeta,s)
\frac{\tau(s)}{\tau(\zeta)}\text{d}\zeta\right) [n(s) - n^{k-1}(s)]\frac{\text{d}s}{\tau(s)}
\\
& \le \mathcal{L} \int_0^\infty [n(s)\!-\!n^{k-1}(s)]\frac{\text{d}s}{\tau(s)}
\le \mathcal{L}^k \int_0^\infty [n(s)\!-\!n^0(s)]\frac{\text{d}s}{\tau(s)}
\\
& = \mathcal{L}^k\int_0^\infty n(s)\frac{\text{d}s}{\tau(s)}
\le A\mathcal{L}^k\int_0^\infty e^{-W(s)}\frac{\text{d}s}{\tau(s)}
\le A\ell e^{\vert W_\text{min}\vert}\mathcal{L}^k.
\end{align*}
This proves the first half of Theorem \ref{thrm-exp-conv}.

This result is easily transformed into an exponential convergence statement
on $F^k_\pm$. Indeed, by Lemma \ref{lem-monoton}, we know that $F_\pm^k$
is nondecreasing in $k$ and converges pointwise to $F_\pm$, so that
\begin{align*}
& \int_0^\infty\!\!\!\int_{-\infty}^\infty \big[\vert F_+(\zeta,\varepsilon)-F^k_+(\zeta,\varepsilon)\vert
+\vert F_-(\zeta,\varepsilon)-F^k_-(\zeta,\varepsilon)\vert\big]
\frac{\bm{1}_{\varepsilon>W(\zeta)} \text{d}\varepsilon \text{d}\zeta}
{\tau(\zeta)\sqrt{2[\varepsilon-W(\zeta)]}}
\\
& \quad =\int_0^\infty\!\!\!\int_{W(\zeta)}^\infty \big[F_+(\zeta,\varepsilon)-F^k_+(\zeta,\varepsilon)
+F_-(\zeta,\varepsilon)-F^k_-(\zeta,\varepsilon)\big]
\frac{\text{d}\varepsilon \text{d}\zeta}{\tau(\zeta)\sqrt{2[\varepsilon-W(\zeta)]}}
\\
& \quad =\int_0^\infty [n(\zeta)-n^k(\zeta)]\frac{\text{d}\zeta}{\tau(\zeta)}
\le A\ell e^{\vert W_\text{min}\vert}\mathcal{L}^k.
\end{align*}
This completes the proof of Theorem \ref{thrm-exp-conv}.

\section{Numerical results}\label{sec:numerical}

In our previous paper \cite{AGK22}, the problem \eqref{HS-1} was solved approximately
to construct models of the boundary condition \eqref{albedo-2}
and also numerically to assess the constructed models. 
The actual numerical analysis was basically carried out for \eqref{HS-3}
using a finite-difference method, the details of which are found
in Appendix C in \cite{AGK22}.

In \cite{AGK22}, the presented
results were mostly regarding the output distribution function
$\hat{f} (\zeta\to\infty,\hat{c}_z>0)$
in response to the input $\hat{f} (\zeta \to \infty, \hat{c}_z<0)$ in \eqref{albedo-1},
since the attention was focused on the boundary condition \eqref{albedo-2} for the
Boltzmann equation. In this section, using the same numerical scheme as in
\cite{AGK22}, we give some numerical results that
visualize the mathematical properties given in Sec.~\ref{sec:math}
as well as that demonstrate the behavior of the gas in the physisorbate layer.

\subsection{Preliminaries}

In order to carry out actual numerical computations for the problem \eqref{HS-3}, 
one has to specify the
interaction potential $W(\zeta)$ and the relaxation time $\tau(\zeta)$ explicitly.
Following \cite{AGK22}, we adopt the Lennard-Jones (LJ) (12,\,6)  and (9,\,3)
potentials:
\begin{subequations}\label{LJ-dimless}
\begin{align}
& W (\zeta) = 4 \kappa  \left( \frac{1}{\zeta^{12}} 
- \frac{1}{\zeta^6} \right), 
\label{LJ-dimless-a} \\
& W (\zeta) = \frac{3\sqrt{3}}{2} \kappa \left( \frac{1}{\zeta^9} 
- \frac{1}{\zeta^3} \right),
\label{LJ-dimless-b}
\end{align}
\end{subequations}
and the relaxation times of algebraic and exponential type:
\begin{subequations}\label{tau-dimless}
\begin{align}
& \tau (\zeta) = \kappa_\tau \left( 1 + \frac{\sigma}{\nu} \zeta \right)^\nu,
\label{tau-dimless-a} \\
& \tau (\zeta) = \kappa_\tau \exp (\sigma \zeta),
\label{tau-dimless-b}
\end{align}
\end{subequations}
where $\kappa$, $\kappa_\tau$, $\nu$, and $\sigma$ are parameters.
Note that \eqref{LJ-dimless-a}, \eqref{LJ-dimless-b}, \eqref{tau-dimless-a},
and \eqref{tau-dimless-b}
are in dimensionless form and correspond to (85), (88), (91), and (96) in \cite{AGK22}.
The reader is referred to \cite{AGK22} for the related quantities; for instance,
$\zeta_a(\varepsilon)$ and $\zeta_b(\varepsilon)$ for \eqref{LJ-dimless-a}
are given by (86) in \cite{AGK22} and those for \eqref{LJ-dimless-b} are
given by (89) there. It is noted that the LJ(9,\,3) potential \eqref{LJ-dimless-b}
is more realistic as a potential of interactions between
a gas molecule and a crystal surface \cite{R84,B00,L08}.
In fact, it results from a continuous model of the crystal
after volume integration.

In addition, the input velocity distribution $F_\infty (c_z)$ in
\eqref{HS-2-d} should be specified. Here, we assume the following
shifted Maxwellian:
\begin{align}\label{F-infty}
F_\infty (c_z) = \frac{1}{\sqrt{2\pi T_\infty}} \exp \left(
- \frac{(c_z - v_{z \infty})^2}{2T_\infty} \right),
\end{align}
where $v_{z \infty}$ and $T_\infty$ are the parameters to be specified.

In this Sec.~\ref{sec:numerical}, the parameters $\kappa$ in \eqref{LJ-dimless}
and $\kappa_\tau$ and $\sigma$ in \eqref{tau-dimless} are set to be
\begin{align}\label{parameters}
(\kappa, \kappa_\tau, \sigma) = (1,1,1)
\end{align}
as in \cite{AGK22}.

\subsection{Monotonicity and exponential convergence with respect to $k$}

The numerical method used in \cite{AGK22} is based on an iteration
scheme essentially the same as \eqref{HS-3-it} except that the superscript $k$
on the right-hand side of \eqref{HS-3-it-e} is replaced by $k-1$.
We also check that the numerical solution enjoys the mathematical properties
established in Sec.~\ref{sec:math}.

In the numerical computation in \cite{AGK22}, initial values different
from \eqref{HS-3-it-init} were used for the iteration. In order to mimic
the mathematical proofs in Sec.~\ref{sec:math}, we redo the computation
using the scheme in \cite{AGK22} but starting from the initial values
\eqref{HS-3-it-init}, i.e., $F_\pm^0 = n^0 = 0$.

The computation is performed in the cases summarized in Table \ref{tab1}.
To be more specific, ``LJ(12,\,6)'' indicates \eqref{LJ-dimless-a}, and
``LJ(9,\,3)'' \eqref{LJ-dimless-b}; ``algebraic'' indicates \eqref{tau-dimless-a}
(with $\nu=4$ and $7$), and ``exponential'' \eqref{tau-dimless-b}; and
the parameters $T_\infty$ and $v_{z \infty}$ are chosen as shown in the
table.
The case (viii) corresponds to the equilibrium solution.
Recall that the parameters $\kappa$, $\kappa_\tau$, and $\sigma$
are set as \eqref{parameters}.

\begin{table}[h]
\begin{center}
\begin{minipage}{0.7\textwidth}
\caption{Computational cases and values of $a$ and $b$.}\label{tab1}%
\begin{tabular}{@{}lllll@{}}
\toprule
 & & $(T_\infty,\,v_{z\infty})$ & $a$ & $b$ \\
\midrule
(i)  &LJ(12,6),\ algebraic ($\nu=7$)& $(1,\,-0.5)$  & $0.06556$ & $-1.127$ \\
(ii)  &LJ(12,6),\ exponential      & $(1,\,-0.5)$  & $0.06226$ & $-1.221$ \\
(iii)  &LJ(9,3),\ algebraic ($\nu=4$) & $(1,\,-0.5)$  & $0.08827$ & $-0.7568$ \\
(iv)  &LJ(9,3),\ exponential       & $(1,\,-0.5)$  & $0.08193$ & $-0.9159$ \\
(v)  &LJ(9,3),\ algebraic ($\nu=4$) & $(1,\,0.5)$   & $0.08827$ & $-1.932$ \\
(vi)  &LJ(9,3),\ algebraic ($\nu=4$) & $(0.6,\,0)$   & $0.08827$ & $-1.466$ \\
(vii)  &LJ(9,3),\ algebraic ($\nu=4$) & $(0.6,\,-0.5)$& $0.08827$ & $-0.8125$ \\
(viii)  &LJ(9,3),\ algebraic ($\nu=4$) & $(1,\,0)$     & $0.08827$ & $-1.266$ \\
\botrule
\end{tabular}
\end{minipage}
\end{center}
\end{table}

Figures \ref{fig3} and \ref{fig4} show the difference $n^k(\zeta) - n^{k-1}(\zeta)$ versus $k$
at four different positions $\zeta$ in the semi-logarithmic scale.  
Figure~\ref{fig3} contains the cases (i)--(iv) in Table \ref{tab1}, whereas
Fig.~\ref{fig4} the cases (v)--(viii) there.
The green, red, blue, and orange solid lines indicate the results at $\zeta=1.371$,
$1.122\, (=\zeta_\text{min})$, $1$, and $0.934$, respectively, in panels (a)
and (b) in Fig.~\ref{fig3} and indicate the results at $\zeta=2.293$, $1.201\, (=\zeta_\text{min})$,
$1$, and $0.901$, respectively, in panels (c) and (d) in Fig.~\ref{fig3}
as well as in all panels in Fig.~\ref{fig4}.
In Fig.~\ref{fig3}, the pair of panels (a) and (b) and that of
panels (c) and (d) show the effect of different relaxation times $\tau(\zeta)$,
whereas the pair of panels (a) and (c) and that of panels (b) and (d) show the
effect of different potentials $W(\zeta)$. Figure \ref{fig4} shows the effect
of the difference in the parameters $(T_\infty, v_{z \infty})$ in the input
velocity distribution \eqref{F-infty}.

In each panel in Figs.~\ref{fig3} and \ref{fig4}, the four solid lines
corresponding to the four different positions
seem to be straight and parallel for large $k$. 
Therefore, it is likely
that $n^k(\zeta) - n^{k-1}(\zeta)$ for large $k$ is expressed in the following form:
\begin{align*}
n^k(\zeta) - n^{k-1}(\zeta) = e^{-ak + b}, \qquad (\text{for large }k),
\end{align*}
where $a$ and $b$ depend on the parameters, but $a$ is independent 
of $\zeta$. The values of $a$ and $b$ determined by the least
square fitting using the numerical data at $\zeta=\zeta_\text{min}$ for
$k \ge 50$ are shown in Table \ref{tab1} for each case, and
the line $C e^{-ak}$ with $a$ in Table \ref{tab1} and an appropriate constant
$C$ is shown by the black dashed line in each panel.

It is seen from Figs.~\ref{fig3} and \ref{fig4} that $n^k(\zeta) - n^{k-1}(\zeta)$ is
always positive. In addition, it is likely from these figures and the above discussion
that $n^k(\zeta) - n^{k-1}(\zeta)$ decreases exponentially in $k$ with
the convergence rate independent of $\zeta$. 
This also indicates that $n^k(\zeta)$ is likely to converge
to $n(\zeta)$ exponentially fast in $k$ with the same uniform
convergence rate, i.e,
\begin{align*}
n(\zeta) - n^k(\zeta) \simeq N(\zeta) e^{-ak},
\end{align*}
for large $k$, where $N(\zeta)$ is an appropriate positive function of $\zeta$.
These observations numerically confirm a part of the statements in
Lemma \ref{lem-monoton} and Theorem \ref{thrm-exp-conv},
i.e., the fact that the sequence
$\{n^k (\zeta)\}$ increases monotonically in $k$ and converges
exponentially fast in $k$. Table \ref{tab1} shows that the convergence rate
is relatively small and is less than $0.1$.
Although Theorem \ref{thrm-exp-conv}
shows the exponential convergence of $n^k$ in a $L^1$ norm, the numerical
result suggests a pointwise convergence in $k$ with
a convergence rate uniform in $\zeta$.

\begin{figure}[htb]
\begin{center}
\includegraphics[width= 0.9\textwidth]{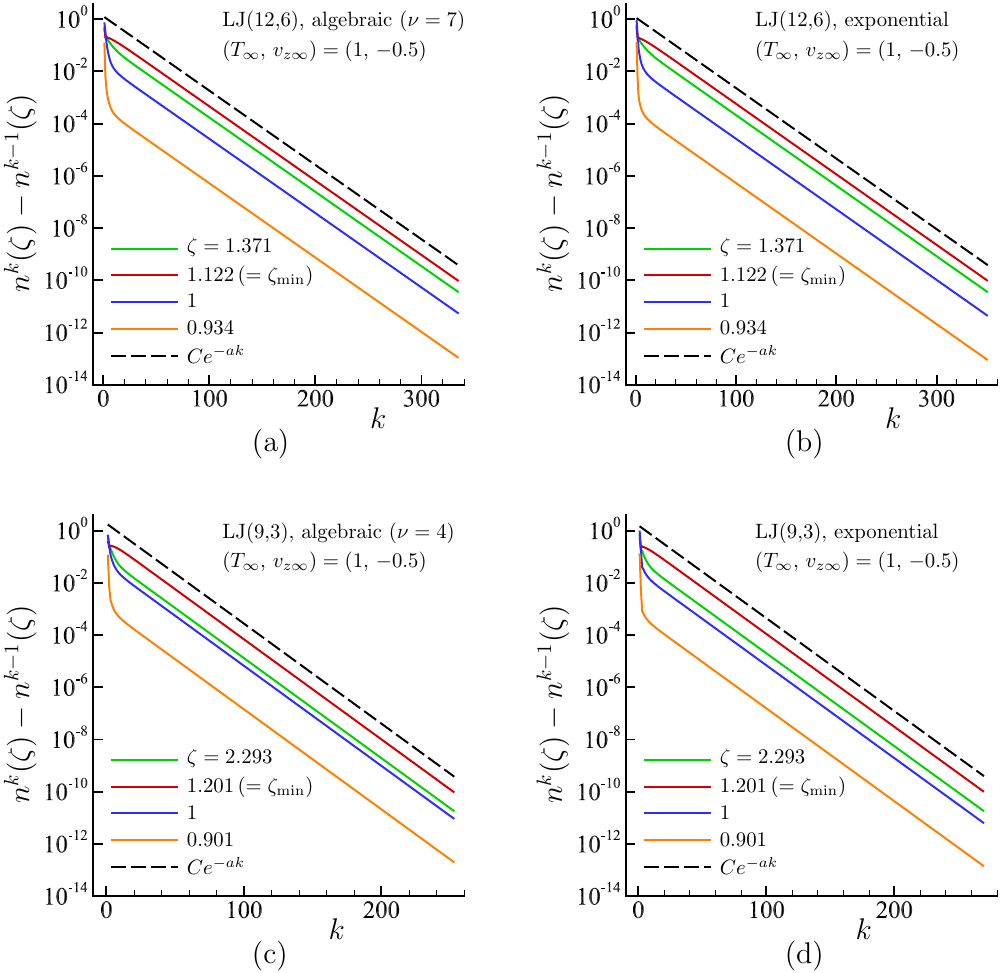}
\caption{\label{fig3}
The difference
$n^k(\zeta) - n^{k-1}(\zeta)$ versus $k$ at four different positions.
Panels (a)--(d) correspond to the cases (i)--(iv) in Table \ref{tab1}.
The green, red, blue, and orange solid lines indicate the results at $\zeta=1.371$,
$1.122\, (=\zeta_\text{min})$, $1$, and $0.934$, respectively, in panels (a)
and (b) and indicate the results at $\zeta=2.293$, $1.201\, (=\zeta_\text{min})$,
$1$, and $0.901$, respectively, in panels (c) and (d). 
The function $C e^{-ak}$ with $a$ in Table \ref{tab1} and an appropriate constant
$C$ is shown by the black dashed line.
}
\end{center}
\end{figure}
%
\begin{figure}[htb]
\begin{center}
\includegraphics[width= 0.9\textwidth]{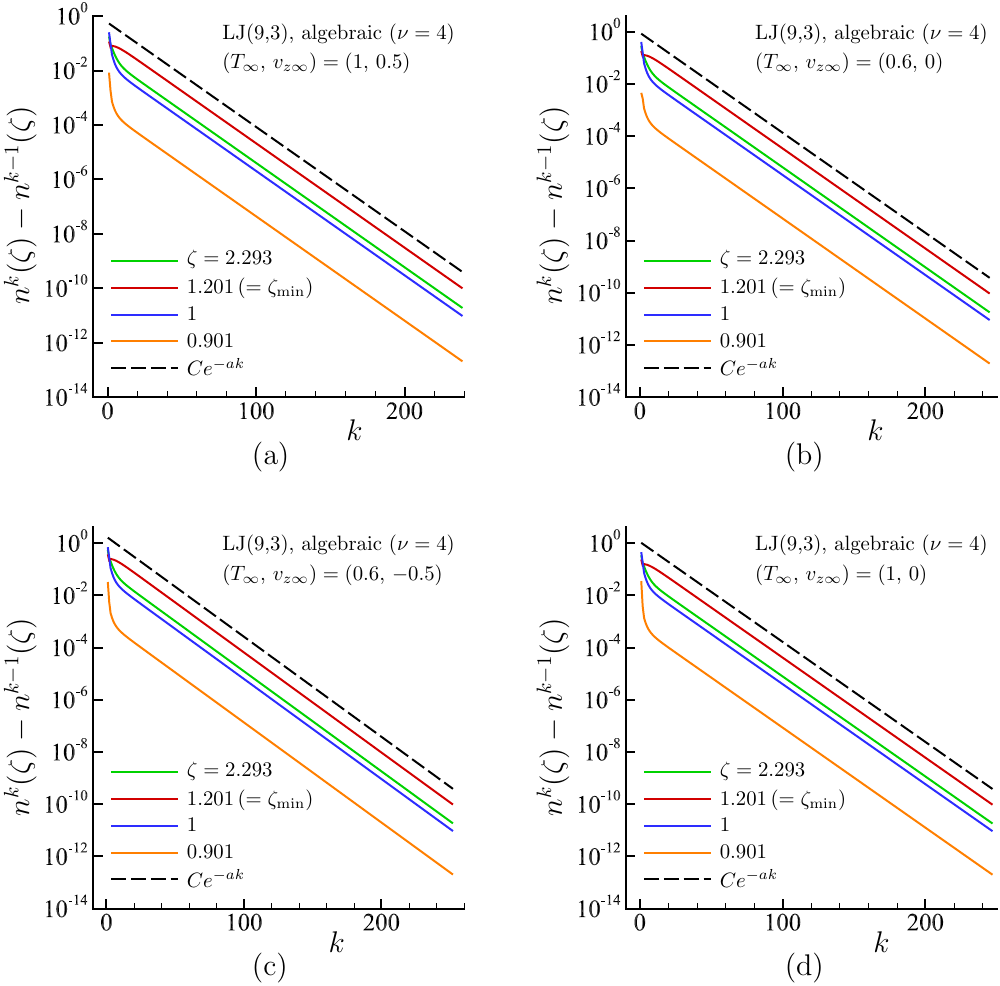}
\caption{\label{fig4}
The difference $n^k(\zeta) - n^{k-1}(\zeta)$ versus $k$ at four different positions.
Panels (a)--(d) correspond to the cases (v)--(viii) in Table \ref{tab1}.
The green, red, blue, and orange solid lines indicate the results at $\zeta=2.293$,
$1.201\, (=\zeta_\text{min})$,
$1$, and $0.901$, respectively. 
The function $C e^{-ak}$ with $a$ in Table \ref{tab1} and an appropriate constant
$C$ is shown by the black dashed line.
}
\end{center}
\end{figure}

\subsection{Behavior in the physisorbate layer}

In this section, we give some numerical results for the behavior
of the gas and physisorbed molecules inside the physisorbate layer.

\subsubsection{Profiles of macroscopic quantities}

We recall that $n (\zeta)$ indicates the dimensionless number density of the gas
molecules at the zeroth order in $\epsilon$ because of the notation 
agreement in Secs.~\ref{subsec:physisorbate} and \ref{subsec:half-space}
(that is, the superscript $\langle 0 \rangle$ as well as the hat has
been omitted) and is expressed as \eqref{HS-2-b} in terms of
the reduced velocity distribution function $F(\zeta, c_z)$.

Figure \ref{fig5} shows $n (\zeta)$ versus $\zeta$ in the cases (i)--(viii) in Table \ref{tab1};
panel (a) contains the cases (i)--(iv), and panel (b) the cases (v)--(viii).
The number density naturally increases in the physisorbate layer and
exhibits the maximum concentration around $\zeta =\zeta_\text{min}$.
The result for the case (viii) recovers the equilibrium
solution $n(\zeta) = e^{-W(\zeta)}$.
In this way, the profiles of $n(\zeta)$ visualizes the
physisorbate layer.

\begin{figure}[htb]
\begin{center}
\includegraphics[width= 0.9\textwidth]{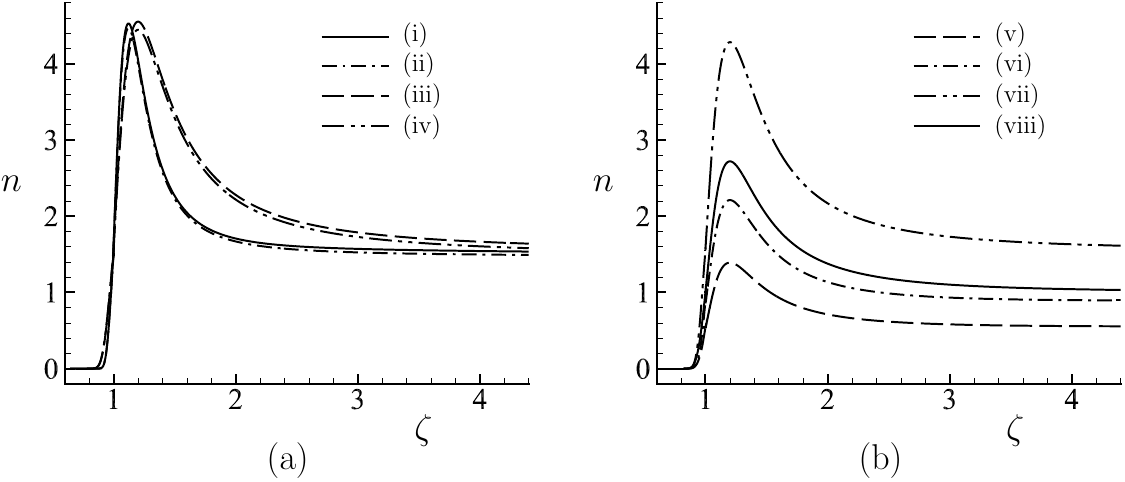}
\caption{\label{fig5} The number density $n$ versus $\zeta$.
(a) cases (i)--(iv)
in Table \ref{tab1}, (b) cases (v)--(viii) there. }
\end{center}
\end{figure}

Let us denote by $c^\star \hat{\bm{v}}$ and $T_\text{w} \hat{T}$ the
(dimensional) flow (or macroscopic) velocity and the (dimensional) temperature,
respectively. We follow the notation 
agreement in Secs.~\ref{subsec:physisorbate} and \ref{subsec:half-space} 
and regard $\bm{v}$ and $T$ as the dimensionless flow velocity
and temperature at the zeroth order in $\epsilon$. Then, they are expressed as
\begin{align}
\bm{v} = \frac{1}{n} \int_{{\bf{R}}^3} \bm{c} f \text{d}\bm{c}, \qquad
T = \frac{1}{3 n} \int_{{\bf{R}}^3} \vert \bm{c} - \bm{v} \vert^2 f \text{d}\bm{c}.
\end{align}
It follows from the particle conservation \eqref{cons-c} that
\begin{align*}
v_z = \frac{1}{n} \int_{{\bf{R}}^3} c_z f \text{d}\bm{c}
= \frac{1}{n} \int_{-\infty}^\infty c_z F(\zeta, c_z) \text{d}c_z = 0,
\end{align*}
at any $\zeta$, that is, there is no macroscopic motion of the gas molecules in the normal
direction in the physisorbate layer. Then, $T$ is expressed as
\begin{align*}
T = \frac{1}{3 n} \int_{{\bf{R}}^3} [c_z^2 + (c_x-v_x)^2 + (c_y-v_y)^2] f \text{d}\bm{c}.
\end{align*}
This indicates that the temperature $T$ is obtained not by the reduced distribution
function $F(\zeta, c_z)$ but by the full distribution $f(\zeta, \bm{c}_\parallel, c_z)$.
However, once $F(\zeta, c_z)$ is obtained, $f$ can be reconstructed as described
in the last paragraph in Sec.~\ref{subsec:half-space}. Here, in order to simplify the
presentation, we consider, instead of the full temperature $T$, the {\it normal}
temperature $T_\perp$ defined by
\begin{align*}
T_\perp (\zeta) = \frac{1}{n} \int_{{\bf{R}}^3} c_z^2 f(\zeta, \bm{c}_\parallel, c_z) \text{d} \bm{c}
= \frac{1}{n} \int_{-\infty}^\infty c_z^2 F(\zeta, c_z) \text{d} c_z.
\end{align*}

The profile of $T_\perp (\zeta)$ is shown in Fig.~\ref{fig6}, where panel (a) contains
the cases (i)--(iv) in Table \ref{tab1} and panel (b) the cases (v)--(viii) there.
Since $F(\zeta, c_z) \to 0$ as $\zeta \to 0$, the normal temperature
$T_\perp$ becomes
$0/0$ as $\zeta \to 0$ and does not make sense in the numerical point of view.
Therefore, $T_\perp$ is shown
only for $\zeta \ge 1$ in Fig.~\ref{fig6}. The result for the case (viii) recovers the equilibrium
solution $T_\perp = 1$. The profile for each of the cases (i)--(iv) exhibits a sharp
downward peak close to $\zeta=1$, whereas that for each of the cases (v)--(vii)
exhibits a sharp upward peak there.

\begin{figure}[htb]
\begin{center}
\includegraphics[width= 0.9\textwidth]{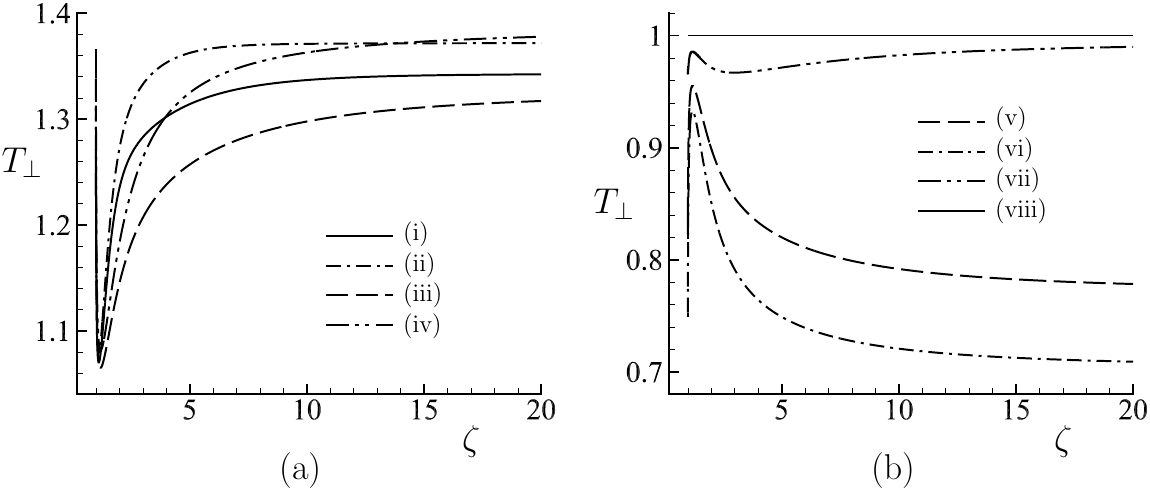}
\caption{\label{fig6} The normal temperature $T_\perp$ versus $\zeta$
for $\zeta \ge 1$.
(a) cases (i)--(iv)
in Table \ref{tab1}, (b) cases (v)--(viii) there.}
\end{center}
\end{figure}

%
\subsubsection{Behavior of velocity distribution function}

Next, we consider the reduced velocity distribution function $F(\zeta, c_z)$.
In Figs.~\ref{fig7}--\ref{fig10}, $F(\zeta, c_z)$ versus $c_z$ is plotted at various values of $\zeta$
for the cases (iii) and (v)--(vii) in Table \ref{tab1}, respectively.
In each figure, panel (a) shows the profiles for $2.239\le\zeta\le\infty$,
(b) for $1.201\le\zeta\le 2.293$, (c) for $1\le\zeta\le 1.201$, and
(d) for $0.901\le\zeta\le 1$. The vertical dotted line indicates the
discontinuities in $F(\zeta, c_z)$.

The profiles of $F(\zeta, c_z)$ in Figs.~\ref{fig7}--\ref{fig10} exhibit discontinuities
at $\varepsilon = c_z^2/2 + W(\zeta) = 0$, namely, at $c_z=\pm \sqrt{-2W(\zeta)}$
(for $\zeta \ge 1$).
The mechanism of generation and propagation of the discontinuity
is explained in detail, for the case (iv) in Table \ref{tab1}, in Sec.~VII of \cite{AGK22}.
Here, we repeat a brief explanation. 

Let us consider the characteristic line
$c_z^2/2 + W(\zeta) = \varepsilon =0$ of \eqref{HS-2-a} for $c_z < 0$.
The cases $\varepsilon=0^+$ and $0^-$ correspond, respectively, to $F_-(\zeta, 0^+)$ and
$F_-(\zeta, 0^-)$ in the $(\zeta, \varepsilon)$ representation.
As $\zeta \to \infty$, $F_-(\zeta, 0^+)$ approaches $F_\infty (0^-)$
because of \eqref{HS-3-d}, whereas $F_-(\zeta, 0^-)$ approaches $F_+(\zeta_b(0^-), 0^-)$
because of \eqref{HS-3-continuity-b}.
Since $F_\infty (0^-)$ and $F_+(\zeta_b(0^-), 0^-)$ are generally different,
$F_- (\zeta, \varepsilon)$ is discontinuous at $\zeta=\infty$ and $\varepsilon=0$,
or equivalently, $F(\zeta, c_z)$ is discontinuous at $\zeta=\infty$ and $c_z=0$.
This discontinuity propagates along the characteristic line $c_z^2/2 + W(\zeta) =0$
for $c_z \le 0$, decaying slowly because of the interaction of gas molecules with
phonons, toward the solid surface and reaches the {\it turning} point
$(\zeta, c_z) = (\zeta_a (0), 0) = (1, 0)$.  Then, it propagates back along the
characteristic line $c_z^2/2 + W(\zeta) =0$ for $c_z >0$, continuing to decay slowly,
toward infinity and
finally reaches infinity. The discontinuity does not enter the range $\zeta < \zeta_a(0) = 1$
because the characteristic line $c_z^2/2 + W(\zeta) =0$ does not enter there.
This behavior is well represented by the profiles in Figs.~\ref{fig7}--\ref{fig10}.
The short vertical dotted-line segment at $c_z=0$ and $\zeta=1$ in 
panels (c) and (d) in each of Figs.~\ref{fig7}--\ref{fig10}, which is upward
in Figs.~\ref{fig7} and \ref{fig10} and downward in Figs.~\ref{fig8} and \ref{fig9},
indicates the height of the discontinuity in $F(\zeta, c_z)$ at $\zeta=1$.

\begin{figure}[htb]
\begin{center}
\includegraphics[width= 0.9\textwidth]{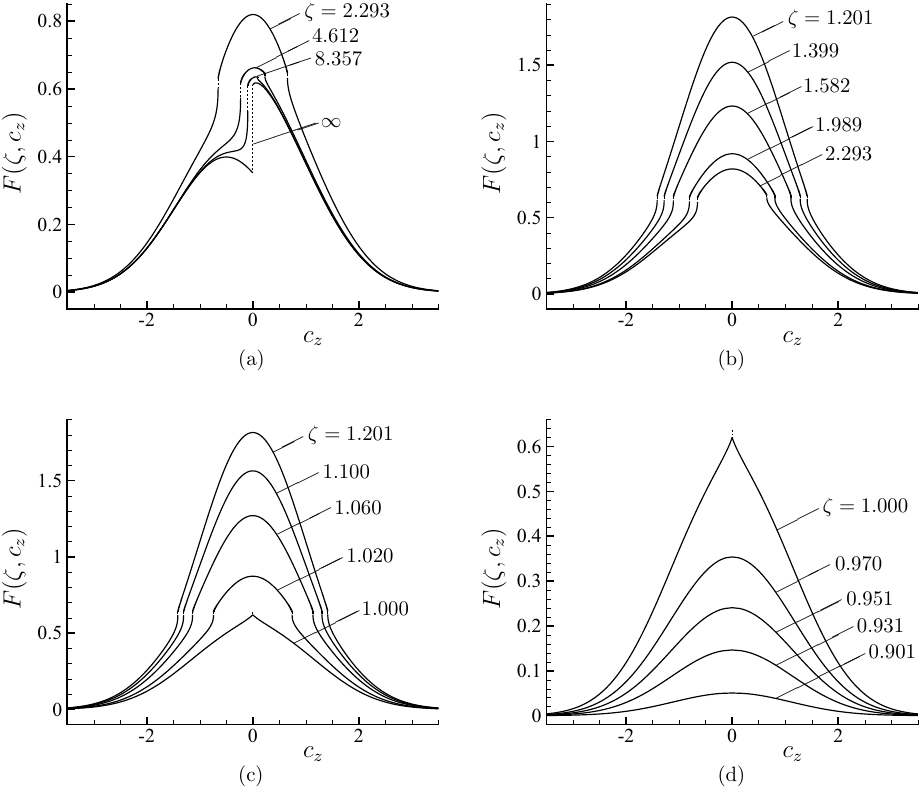}
\caption{\label{fig7} 
The reduced velocity distribution
$F(\zeta,c_z)$ versus $c_z$ in the case (iii) in Table \ref{tab1}. 
(a) $2.239\le\zeta\le\infty$, (b) $1.201\le\zeta\le 2.293$, 
(c) $1\le\zeta\le 1.201$, and (d) $0.901\le\zeta\le 1$.}
\end{center}
\end{figure}
%
\begin{figure}[htb]
\begin{center}
\includegraphics[width= 0.9\textwidth]{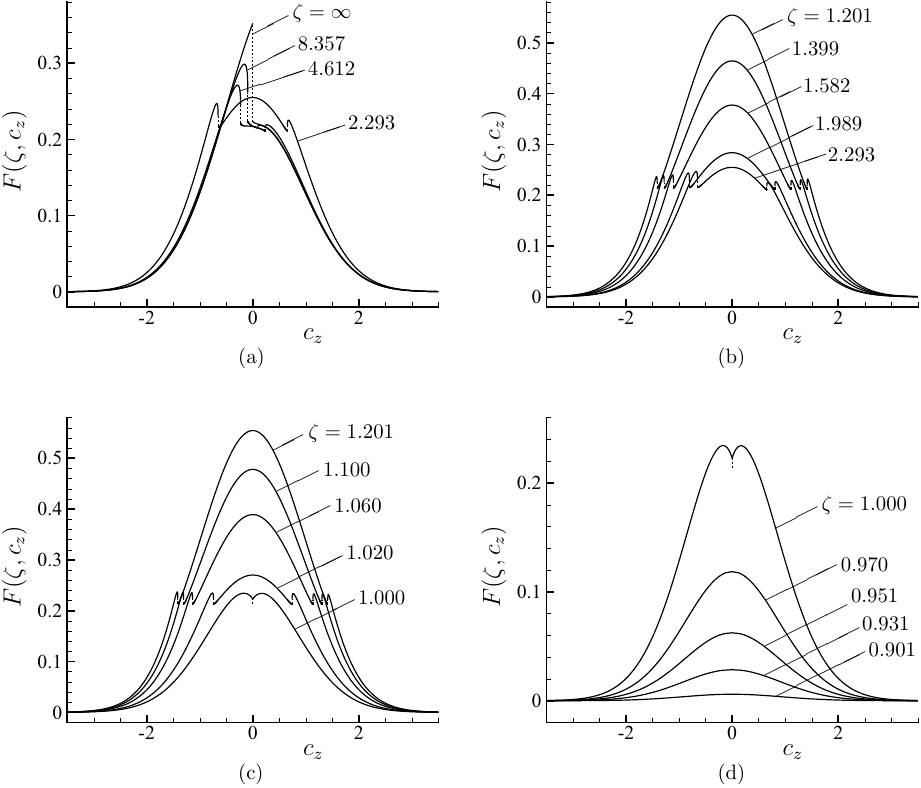}
\caption{\label{fig8} 
The reduced velocity distribution
$F(\zeta,c_z)$ versus $c_z$ in the case (v) in Table \ref{tab1}. 
(a) $2.239\le\zeta\le\infty$, (b) $1.201\le\zeta\le 2.293$, 
(c) $1\le\zeta\le 1.201$, and (d) $0.901\le\zeta\le 1$.}
\end{center}
\end{figure}
%
\begin{figure}[htb]
\begin{center}
\includegraphics[width= 0.9\textwidth]{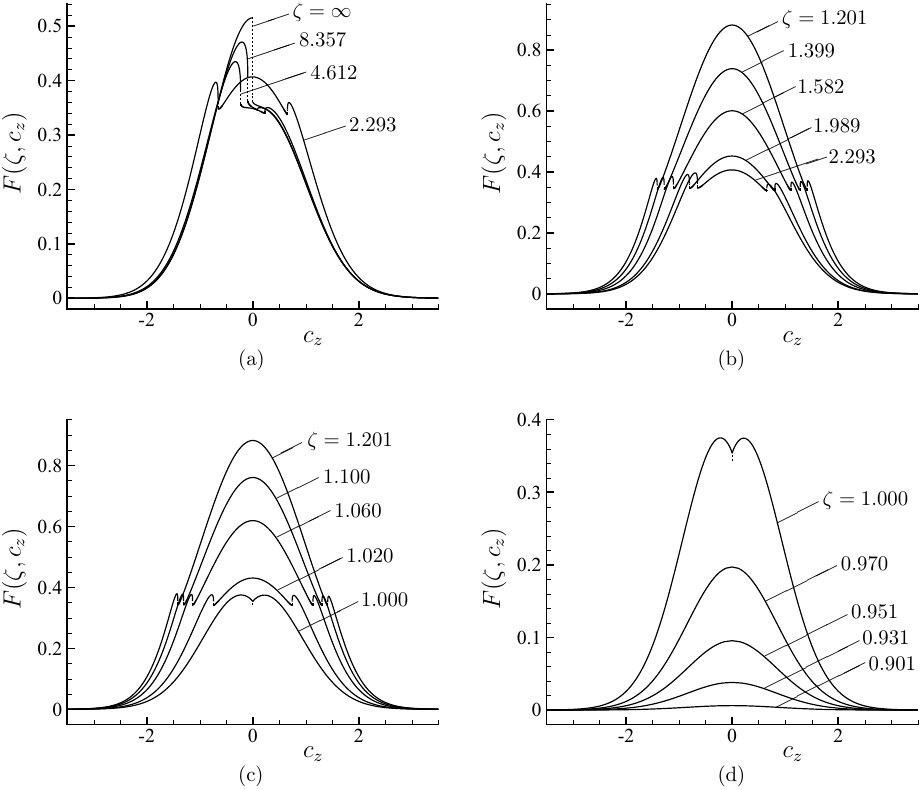}
\caption{\label{fig9} 
The reduced velocity distribution
$F(\zeta,c_z)$ versus $c_z$ in the case (vi) in Table \ref{tab1}. 
(a) $2.239\le\zeta\le\infty$, (b) $1.201\le\zeta\le 2.293$, 
(c) $1\le\zeta\le 1.201$, and (d) $0.901\le\zeta\le 1$.}
\end{center}
\end{figure}
%
\begin{figure}[htb]
\begin{center}
\includegraphics[width= 0.9\textwidth]{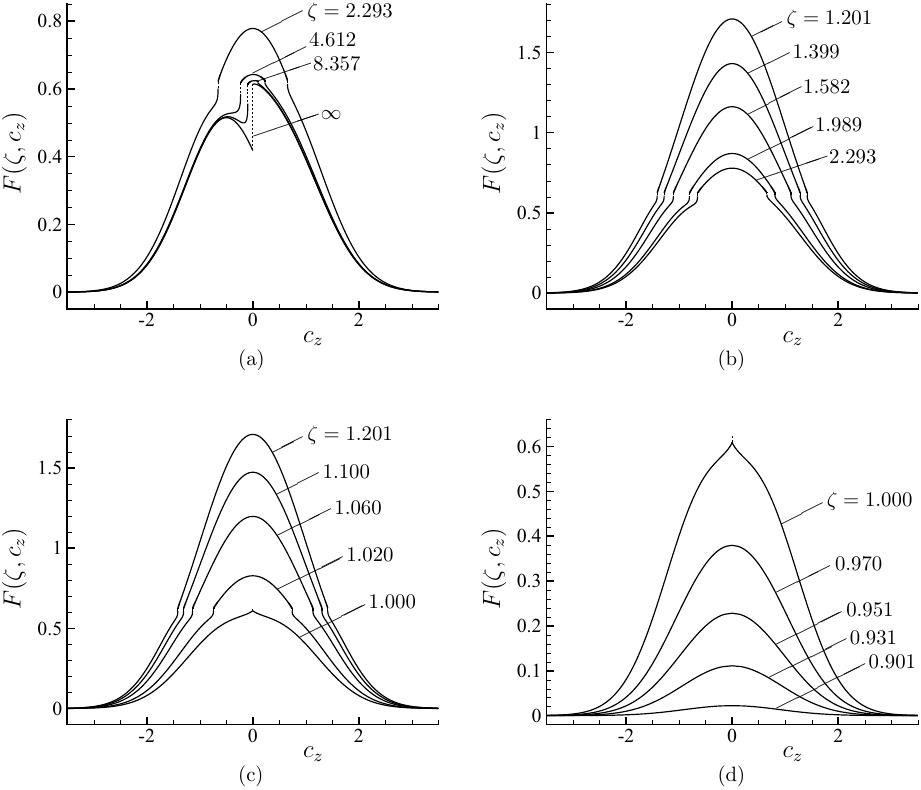}
\caption{\label{fig10} 
The reduced velocity distribution
$F(\zeta,c_z)$ versus $c_z$ in the case (vii) in Table \ref{tab1}. 
(a) $2.239\le\zeta\le\infty$, (b) $1.201\le\zeta\le 2.293$, 
(c) $1\le\zeta\le 1.201$, and (d) $0.901\le\zeta\le 1$.}
\end{center}
\end{figure}

\subsection{Boundary condition for the Boltzmann equation}\label{subsec:BC}

In this paper, special attention has been focused on the half-space problem
\eqref{HS-2} for the reduced distribution function $F(\zeta, c_z)$ for
the physisorbate layer. The main purpose was to establish some rigorous
mathematical results summarized in Theorems \ref{thm-exist} and \ref{thrm-exp-conv}
for the problem \eqref{HS-2}. In addition, some numerical results
visualizing rigorous mathematical properties as well as showing
the behavior of the gas inside the physisorbate layer were presented.
Here, recalling that the original aim of considering the half-space
problem \eqref{HS-2} or \eqref{HS-1} was to establish the boundary
condition for the Boltzmann equation, we touch on this aspect of the
problem.

In \cite{AGK22}, analytical models of the boundary conditions for the
Boltzmann equation were proposed. The models were constructed
as the first and second approximations for the iteration scheme for
\eqref{HS-1}. The scheme is essentially the same as the scheme
\eqref{HS-3-it-intg} for \eqref{HS-3} [see (51) in \cite{AGK22}], so that
it is omitted here. The point is that the iteration starts not from
zero [cf.~\eqref{HS-3-it-init}] but from the equilibrium solution
\begin{align*}
f^{0} = \beta \frac{1}{(2\pi)^{3/2}} \exp \left( -\frac{\vert \bm{c} \vert^2}{2} - W(\zeta) \right),
\end{align*}
where $f^{0}$ indicates the zeroth guess for the approximating sequence
$\{f^k\}$ ($k=1,\, 2,\, \dots$) obtained by the iteration scheme
mentioned above, and $\beta$ is a constant to be determined in
such a way that the particle conservation \eqref{cons-c} is satisfied in
each $f^k$.

The model of the boundary condition \eqref{albedo-2}, based on the first
iteration, is obtained in the following form \cite{AGK22}:
\begin{align}\label{BC-1st}
& f_\text{g} (z, c_z) = [1 - \alpha (c_z^2)] f_\text{g} (z, -c_z) + \alpha(c_z^2)
\beta\, (2\pi)^{-3/2} \exp \left( - \vert \bm{c} \vert^2/2 \right),
\nonumber \\
& \qquad \qquad \qquad \qquad \qquad \qquad \qquad \qquad \qquad \qquad
\text{for} \;\; c_z > 0, \;\; \text{at}\;\; z=0,
\end{align}
where
\begin{align}
& \alpha (c_z^2)
= 1 - \big[ \theta \big( \zeta_a (c_z^2/2), \infty; c_z^2/2 \big) \big]^2 
\nonumber \\
& \qquad \; = 1 - \exp \left( - \sqrt{2} \int_{\zeta_a (c_z^2/2)}^\infty
\frac{\text{d} \xi}{\tau (\xi) \sqrt{ c_z^2/2 - {W}(\xi)}} \right),
\label{accomm} \\
& \beta = - \sqrt{2\pi} 
\left[ \int_0^\infty c_z \alpha (c_z^2) \exp \left( - c_z^2/2
\right) \text{d} c_z \right]^{-1}
\int_{c_z < 0} c_z \alpha (c_z^2) f_\text{g} (0,\, c_z) \text{d} \bm{c}.
\label{beta}
\end{align}
Here, $\zeta_a (c_z^2/2)$ is the solution of $W (\zeta) = c_z^2/2$,
and note that $0 < \alpha (c_z^2)  < 1$ holds.
In addition, it should be recalled that \eqref{BC-1st}--\eqref{beta}
are dimensionless and that the arguments $t$, $\bm{x}_\parallel$, and
$\bm{c}_\parallel$ are omitted in $f_\text{g}$ in \eqref{BC-1st} and \eqref{beta}.
The correspondence of the notation here and in \cite{AGK22} is as follows:
\begin{align*}
& \big( f_\text{g},\, \alpha,\, \beta,\, \bm{x}=(x, y, z),\, \bm{c}=(c_x, c_y, c_z),\, \tau,\, W \big)
\;\; \mbox{(here)}
\\
& \qquad \qquad \Longleftrightarrow 
\big( \hat{f}_\text{g},\, \hat{\alpha},\, \hat{\beta},\, \hat{\bm{x}}=(\hat{x}, \hat{y}, \hat{z}),\,
\hat{\bm{c}}=(\hat{c}_x, \hat{c}_y, \hat{c}_z),\, \hat{\tau}_\text{ph},\, \hat{\textsc{w}} \big)
\;\; \mbox{(in \cite{AGK22})}.
\end{align*}
The reader is referred to \cite{AGK22} for the dimensional form of
\eqref{BC-1st}--\eqref{beta} as well as for the generalization, such as the
cases of varying wall temperature and curved boundary.
The model based on the second iteration, which is
also obtained in \cite{AGK22} and detailed in \cite{KAG23}, is less explicit, so that it is omitted here.

As one can see from \eqref{BC-1st}, the dependence on $c_x$ and $c_y$ of the
reflected distribution $f_\text{g} (0, c_z)$ ($c_z > 0$) is determined partially by that
of the incident distribution $f_\text{g} (0, c_z)$ ($c_z < 0$) and partially by
the thermalizing term $\exp (- \vert \bm{c} \vert^2/2)$.
This tendency is more or less the same for the original half-space problem
\eqref{HS-1}. In other words, the problem \eqref{HS-1} determines the
$c_z$-dependence of the reflected molecules crucially but not the $c_x$- and
$c_y$-dependence.

For this reason and in consistency with the main discussions in this paper,
we consider the models of the boundary condition only for the reduced
distribution function  $F(\zeta, c_z)$. Denoting
\begin{align}\label{marginal-Fg}
F_\text{g} (z, c_z) = \int_{-\infty}^\infty \int_{-\infty}^\infty
f_\text{g} (z, \bm{c}_\parallel, c_z) \text{d} c_x \text{d} c_y,
\end{align}
and integrating \eqref{BC-1st} with respect to $c_x$ and $c_y$ each from
$-\infty$ to $\infty$, one obtains
\begin{align}\label{BC-1st-F}
& F_\text{g} (z, c_z) = [1 - \alpha (c_z^2)] F_\text{g} (z, -c_z) + \alpha(c_z^2)
\beta\, (2\pi)^{-1/2} \exp \left( - c_z^2/2 \right),
\nonumber \\
& \qquad \qquad \qquad \qquad \qquad \qquad \qquad \qquad \qquad \qquad
\text{for} \;\; c_z > 0, \;\; \text{at}\;\; z=0,
\end{align}
where $\alpha (c_z^2)$ is given by \eqref{accomm}, and $\beta$ is recast as
\begin{align}\label{beta-2}
\beta = - \sqrt{2\pi} 
\left[ \int_0^\infty c_z \alpha (c_z^2) \exp \left( - c_z^2/2
\right) \text{d} c_z \right]^{-1}
\int_{-\infty}^0 c_z \alpha (c_z^2) F_\text{g} (0,\, c_z) \text{d} c_z.
\end{align}

Figure \ref{fig11} shows the reduced velocity distribution for the reflected molecules
$F_\text{g} (0, c_z)$ ($c_z > 0$) in response to that for the incident molecules
$F_\text{g} (0, c_z)$ ($c_z < 0$) when the latter is given by \eqref{F-infty},
i.e.,
\begin{align}\label{F-input}
F_\text{g} (0, c_z) = \frac{1}{\sqrt{2\pi T_\infty}} \exp \left(
- \frac{(c_z - v_{z \infty})^2}{2T_\infty} \right), \quad
\text{for}\;\; c_z < 0,
\end{align}
for the cases (iii) and (v)--(vii) in Table \ref{tab1}; panels 
(a), (b), (c), and (d) correspond
to the cases (iii), (v), (vi), and (vii), respectively. In the figure, 
the thin line indicates the result based on the first iteration model,
i.e., \eqref{BC-1st-F} with
\eqref{accomm} and \eqref{beta-2}, the thick line for $c_z>0$
that based on the
second iteration model (see \cite{AGK22,KAG23} for its form), and the small circles
that based on the numerical solution taken from Figs.~\ref{fig7}--\ref{fig10}, i.e., $F(\zeta, c_z)$
for $\zeta = \infty$ for $c_z > 0$.
The thick line for $c_z < 0$ indicates the incident distribution
\eqref{F-input}. At least in the cases (iii) and (v)--(vii) in Table \ref{tab1}, the model
based on the second iteration shows very good agreement with the numerical
solution.

\begin{figure}[htb]
\begin{center}
\includegraphics[width= 0.9\textwidth]{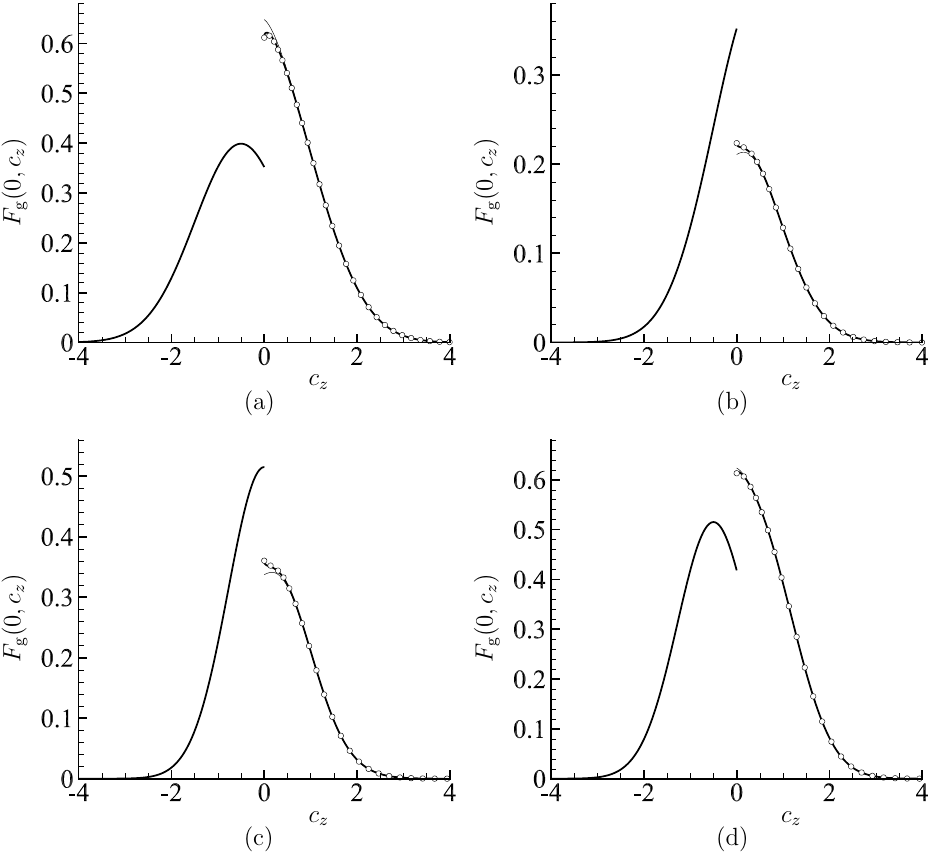}
\caption{\label{fig11} 
$F_\text{g} (0, c_z)$ versus $c_z$. 
Panels (a), (b), (c), and (d) correspond to the cases (iii), (v), (vi), and (vii) 
in Table \ref{tab1}, respectively. 
The thin line indicates the result based on the first iteration model, i.e.,
\eqref{BC-1st-F} with \eqref{accomm} and \eqref{beta-2}, 
the thick line for $c_z>0$ that based on the second iteration model (see \cite{AGK22,KAG23} for its form), 
and the circles that based on the numerical solution.
The thick line for $c_z < 0$ indicates the incident distribution \eqref{F-input}.}
\end{center}
\end{figure}


\section{Concluding remarks}\label{sec:remarks}

The present study concerns a kinetic model of gas-surface interactions and
resulting boundary conditions for the Boltzmann equation on a solid surface \cite{AGK22}.
In the process of the construction of the boundary conditions, a half-space problem of
a kinetic equation describing the behavior of the gas molecules in a thin layer on the solid
surface (physisorbate layer), in which the molecules are subject to an attractive-repulsive
potential and interacting with phonons, plays a crucial role. To be more specific, the
solution of this half-space problem establishes the relation between the velocity distribution
for the incident molecules and that for the outgoing molecules at infinity, and this relation
is nothing but the boundary condition for the Boltzmann equation that is valid outside
the physisorbate layer. In \cite{AGK22}, this fact was clarified by a formal asymptotic
analysis of the kinetic model for gas-surface interactions, and the half-space problem
was solved approximately and numerically to establish the boundary condition for the
Boltzmann equation.

In the present paper, we have deepened the analysis and established
rigorously the essential mathematical properties of the physisorbate-layer problem.
The results are summarized in Theorems \ref{thm-exist} and \ref{thrm-exp-conv} in
Sec.~\ref{subsec:main}. To be more specific, the existence and uniqueness of the solution
have been established. The existence was proved by using an approximating sequence,
which turned out to be non-decreasing and to have an upper bound. Furthermore, the
approximating sequence was proved to converge to the solution exponentially fast
with an explicit estimate of the convergence rate. 

In addition to the rigorous mathematical discussions in Sec.~\ref{sec:math}, we have also
carried out some numerical computations, using the method introduced in \cite{AGK22},
for some specific potentials and gas-phonon relaxation times, to visualize a part of
the mathematical properties shown in Theorems \ref{thm-exist} and \ref{thrm-exp-conv}
as well as to demonstrate the behavior of the velocity distribution function and the
macroscopic quantities inside the physisorbate layer (Sec.~\ref{sec:numerical}). In this
connection, the analytical models of the boundary condition for the Boltzmann equation,
established in \cite{AGK22} on the basis of first and second approximations for the
half-space problem for the physisorbate layer, were also mentioned (Sec.~\ref{subsec:BC}).

The kinetic approach in \cite{AGK22}, which is based on the kinetic scaling in contrast
to the fluid scaling considered in \cite{AGH16,AG19a,AG19b,AG21}, can be generalized
in various directions. For instance, it would be more practical to consider a confinement
potential with a periodic modulation along the solid surface. It would also be important
to include chemisorption in addition to physisorption and to consider
situations where the phonons are not in equilibrium. Each direction should pose
new mathematical problems for relevant kinetic equations, and the mathematical
analysis performed in this paper would provide good and useful guidelines for them.

\bmhead{Acknowledgments}

This work was supported by a French public grant as part of the
Investissement d'avenir project, reference ANR-11-LABX-0056-LMH,
LabEx LMH and by the grant-in-aid No.~21K18692 from JSPS in Japan.

\vspace*{6mm}
\noindent
Conflicts of Interest: The authors declare no conflict of interest.

\vspace*{2mm}
\noindent
Data Availability Statement: The data that support Figs.~1--11 are available from the
corresponding author, K.A., upon reasonable request. All other data are provided in the paper.


\bibliography{sn-bibliography}


%

\end{document}